\documentclass[12pt,reqno]{amsart}

\usepackage{times}
\usepackage{amsmath,amsfonts, amstext,amssymb,amsbsy,amsopn,amsthm}
\usepackage{dsfont}
\usepackage{esint}
\usepackage{graphicx}   
\usepackage{hyperref}
\usepackage[all]{xy}
\usepackage{color}
\usepackage{tikz}
\usepackage{lineno,hyperref}

\newcommand{\be}{\begin{equation}}
\newcommand{\ee}{\end{equation}}

\newtheorem{theorem}{Theorem}[section]
\newtheorem{lemma}[theorem]{Lemma}

\newtheorem{corollary}[theorem]{Corollary}
\theoremstyle{definition}

\theoremstyle{remark}
\newtheorem{remark}{Remark}[section]

\theoremstyle{remark}

\numberwithin{equation}{section}
\allowdisplaybreaks[4]

\begin{document}

\title[Qualitative properties of dual fractional parabolic equations]{Qualitative properties of solutions for dual fractional nonlinear parabolic equations}

\author{Wenxiong Chen }
\address{Department of Mathematical Sciences, Yeshiva University, New York, NY,  10033 USA}
\email{wchen@yu.edu}

\author{Lingwei Ma}
\address{School of Mathematical Sciences, Tianjin Normal University,
Tianjin, 300387, P.R. China, and Department of Mathematical Sciences, Yeshiva University, New York, NY, 10033 USA}
\email{mlw1103@outlook.com}

\date{\today}

\begin{abstract}
In this paper, we consider the dual fractional parabolic problem
$$\left\{
\begin{array}{ll}
    \partial_t^\alpha u(x,t)+(-\Delta)^su(x,t)=f(u(x,t)) ,~  &\mbox{in}\,\,  \mathbb{R}^n_+\times\mathbb{R}  , \\
  u(x,t)= 0 , ~ &\mbox{in}\,\, (\mathbb{R}^n  \backslash \mathbb{R}^n_+) \times\mathbb{R},
\end{array}
\right.
$$
where $\mathbb{R}^n_+:=\{x\in\mathbb{R}^n\mid x_1>0\}$ is the right half space. We prove that the positive solutions are
strictly increasing in $x_1$ direction without assuming the solutions be bounded.

So far as we know, this is the first paper to explore the monotonicity of possibly unbounded solutions for the nonlocal parabolic problem involving both the fractional time derivative $\partial_t^\alpha$ and the fractional Laplacian $(-\Delta)^s$.
To overcome the difficulties caused by the dual nonlocality in space-time and by the remarkably
weak assumptions on solutions, we introduced several new ideas and our approaches are quite different from those in the previous
literature.
We first establish an unbounded narrow region principle without imposing any decay and boundedness assumptions on the antisymmetric functions at infinity by estimating the nonlocal operator $\partial_t^\alpha+(-\Delta)^s$ along a sequence of suitable auxiliary functions at their minimum points, which is an essential
ingredient to carry out the method
of moving planes at the starting point.
Then in order to remove the decay or bounded-ness assumption on the solutions, we develop a new novel approach lies in establishing the {\em averaging effects} for such nonlocal
operator and apply these {\em averaging effects} twice to guarantee that the plane can be moved all the way to infinity to derive the monotonicity of solutions.

We believe that the new ideas and techniques developed here will become very useful tools in
studying the qualitative properties of solutions, in particular of those unbounded solutions,
for a wide range of fractional elliptic and parabolic problems.

Mathematics Subject classification (2020): 35R11; 35B50; 35K58; 26A33.

Keywords: dual nonlocal parabolic equations; fractional time-diffusion; narrow region principle in unbounded domains; averaging effects; direct method of moving planes; monotonicity. \\
\end{abstract}

\maketitle

\section{Introduction}
\label{s:introduction}
The infancy of fractional calculus dates back to
a letter from L'H\^{o}pital in 1695, in which he asked
Leibniz how to define the derivative $\frac{\operatorname{d}\!^n f(x)}{\operatorname{d}\!x^n}$ when the order $n=\frac{1}{2}$ is not an integer.
Since then, this problem has received extensive interest from many mathematicians such as Riemann, Liouville, Riesz, Marchaud, Caputo and so on. They proposed the definitions of fractional derivatives in different forms. The accompanying surprise is that the fractional derivatives can be used to model many important physical phenomena, thereby considerable attentions have been paid to the study of qualitative properties of solutions to equations involving fractional derivatives.

In this paper, we investigate the nonlocal parabolic equations involving the fractional time derivative and the fractional Laplace operator as follows
\begin{equation}\label{model}
\left\{
\begin{array}{ll}
    \partial_t^\alpha u(x,t)+(-\Delta)^su(x,t)=f(u(x,t)) ,~  &\mbox{in}\,\,  \mathbb{R}^n_+\times\mathbb{R}  , \\
  u(x,t)= 0 , ~ &\mbox{in}\,\, (\mathbb{R}^n  \backslash \mathbb{R}^n_+) \times\mathbb{R},
\end{array}
\right.
\end{equation}
where $\mathbb{R}^n_+:=\{x\in\mathbb{R}^n\mid x_1>0\}$ is the right half space.
The space-time nonlocal equation in \eqref{model} can be seen as a typical model in the continuous time random walks \cite{MK}, which is a generalization of the Brownian random walks formulated as the equation involving the local time derivative. The latter describes the particles experience uncorrelated random displacements at fixed time intervals. The time non-locality explains the history dependence introduced in dynamics by the presence of anomalously large waiting time, and the space non-locality accounts for the existence of anomalously large jumps, such as L\'{e}vy flights connecting distant regions in space.

The fractional time derivative $\partial_t^\alpha$ we consider here is the Marchaud fractional derivative of order $\alpha\in(0,1)$, defined by
\begin{equation}\label{fratime}
\partial_t^\alpha u(x,t):=C_{\alpha}\int_{-\infty}^t \frac{u(x,t)-u(x,\tau)}{(t-\tau)^{1+\alpha}}\operatorname{d}\!\tau,
\end{equation}
which was first introduced by Marchaud in 1927. The normalization positive constant $C_\alpha=\frac{\alpha}{\Gamma(1-\alpha)}$, and $\Gamma$ denotes the Gamma function.
In order to guarantee the singular integral in \eqref{fratime} is well defined, we may assume that $u(x,\cdot)\in C^1(\mathbb{R})\times {\mathcal L}^-_{\alpha}(\mathbb{R})$, where ${\mathcal L}^-_{\alpha}(\mathbb{R})$ is a class of slowly increasing functions given by
$$ {\mathcal L}^-_{\alpha}(\mathbb{R}):=\{u(x,\cdot) \in L^1_{\rm loc} (\mathbb{R}) \mid \int_{-\infty}^t \frac{|u(x,\tau)|}{1+|\tau|^{1+\alpha}}\operatorname{d}\!\tau<+\infty\,\, \mbox{for any} \,\, t\in\mathbb{R}\}.$$
Such fractional time derivative emerges in a variety of physical phenomena, for instance,
particle systems with sticking and trapping phenomena, magneto-thermoelastic heat conduction, plasma turbulence and so on (cf. \cite{DCL1, DCL2, EE}).
The spatial nonlocal pseudo-differential operator in \eqref{model}, the fractional Laplacian $(-\Delta)^s$ is defined as
\begin{equation}\label{fraLap}
(-\Delta)^s u(x,t):=C_{n,s} P.V. \int_{\mathbb R^N} \frac{u(x,t)-u(y,t)}{|x-y|^{n+2s}}\operatorname{d}\!y,
\end{equation}
where $0<s<1$, $C_{n,s}$ is a normalization positive constant and $P.V.$ stands for the Cauchy principal value. We define
$$ {\mathcal L}_{2s}(\mathbb{R}^n):=\{u(\cdot, t) \in L^1_{\rm loc} (\mathbb{R}^n) \mid \int_{\mathbb R^n} \frac{|u(x,t)|}{1+|x|^{n+2s}}\operatorname{d}\!x<+\infty\},$$
then $u\in C^{1,1}_{\rm loc}(\mathbb{R}^n_+)\cap {\mathcal L}_{2s}(\mathbb{R}^n)$ ensures the integrability of \eqref{fraLap}. This fractional diffusion operator is of great interest due to its applications in physics.
To name a few, the fractional Laplacian arises in anomalous diffusion, quasi-geostrophic dynamics, phase transition models, and image reconstruction problems (cf. \cite{AB, BG, CV, GO}). Observe that the definition of the fractional time derivative $\partial_{t}^{\alpha}$ given in \eqref{fratime} looks similar to the one-dimensional fractional Laplacian except that such integral only takes account of the interactions in the past, which is related to the fact that the transport is not time reversible and has a memory effect.
It is worth mentioning that the nonlocal operator $\partial_{t}^{\alpha}+(-\Delta)^s$ can be reduced to the local heat operator $\partial_{t}-\Delta$ as $\alpha\rightarrow1$ and $s\rightarrow1$.

During the last decade, considerable attentions have been paid
to the investigation of the aforementioned space-time nonlocal equation in order to
acquire a clearer understanding of various physical phenomena, and then correspondingly many subtle mathematical problems have appeared. In the context of well-posedness such as the existence, uniqueness and regularity to such nonlocal parabolic equations have been systematically studied in a series of remarkable papers \cite{ A1, A2, ACV} by Caffarelli and his group.
However, so far as we know, there have not been any results on the geometry of solutions for such equations, hence our interest here is to propose a holistic approach to establish the monotonicity of solutions to problem \eqref{model} in a half space.

The method of moving planes introduced by Alexandroff in \cite{H} is a common technique to study the monotonicity of solutions to local elliptic and parabolic equations. However, this approach cannot be applied directly to psuedo-differential equations involving the fractional Laplace operator due to its nonlocality. To overcome this difficulty, an important progress traces back to the pioneering work of Caffarelli and Silvestre \cite{CS}, in which they developed an extension method to reduce the nonlocal problem into a local one in a higher dimensional space. Thereby the traditional method of moving planes designed for local equations can be applied for the extended problem to establish the properties of solutions.
Another useful approach is to turn
the given pseudo-differential equations into their equivalent integral equations, then one can use the method of moving planes in integral forms and the regularity lifting to investigate the properties of
solutions (cf. \cite{CLO1, CLO2, CLO}). These two effective methods
have been employed successfully
to investigate the
elliptic equations involving the fractional Laplacian, and a series of interesting results have been obtained in \cite{BCPS, CFY, CZ, DLL, LZ, MCL, ZCCY} and the references therein. However, the above two methods can only be applied to equations involving the fractional Laplacian,
and sometimes one needs to impose extra conditions on the problems which may not be necessary if
dealing with the fractional equations directly.
Nearly ten years later, a further progress was made
by Chen, Li, and Li \cite{CLL}, who introduced a direct method of moving planes to remove the restrictions and to greatly simplify the proof process. Afterwards, such effective direct method has been widely applied to establish the symmetry, monotonicity, non-existence, and even to obtain estimates in a boundary layer of solutions for various elliptic equations and systems involving the fractional Laplacian, the fully nonlinear nonlocal operators, the fractional $p$-Laplacians as well as the higher order fractional operators, we refer to \cite{CH, CL, CLLg, CLM, CWu1, DQ, MZ1, MZ2, MZ3, ZL} and the references therein.

Very recently, many substantial advances have been made in the symmetry, monotonicity, and non-existence of positive solutions for fractional parabolic equations of type \eqref{model} with the usual local time derivative $\partial_t u(x,t)$, based on the method of moving planes (cf. \cite{CWNH, CWu2, CWW, JW, WaC, WuC} and the references therein).

In contrast, the geometric behavior of solutions to nonlocal parabolic equations with simultaneous presence of the fractional time derivative and the fractional Laplacian are still lacking.

Here we present a simple example to illustrate the essential difference between the local time derivative and the fractional time derivative.
Let $\Omega$ be a bounded domain in $\mathbb{R}^n$ and $[t_1, t_2]$ be an interval in $\mathbb{R}$. A typical maximum principle for a fractional parabolic problem with the local time derivative $\partial_t$ in the parabolic cylinder $\Omega\times(t_1,t_2]$  can be formulated as:

If $u(x,t)$ is a solution of
\begin{equation*}
\left\{
\begin{array}{ll}
    \partial_t u(x,t)+(-\Delta)^su(x,t)\geq 0 ,~   &(x,t) \in  \Omega\times(t_1,t_2]  , \\
  u(x,t)\geq 0 , ~ &(x,t)  \in \Omega^c\times(t_1,t_2] ,\\
  u(x,t_1)\geq 0 , ~ &x  \in \Omega ,
\end{array}
\right.
\end{equation*}
then $u(x,t)\geq 0$ in $\Omega\times(t_1,t_2]$.

Here the initial condition is given at the time moment $t=t_1$.

While the maximum principle involving the fractional time derivative $\partial_t^\alpha$ established in the subsequent section takes the following form.

Assume that $u(x,t)$ is a solution of
\begin{equation}\label{NMP}
\left\{
\begin{array}{ll}
    \partial_t^\alpha u(x,t)+(-\Delta)^su(x,t)\geq 0 ,~   &(x,t) \in  \Omega\times(t_1,t_2]  , \\
  u(x,t)\geq 0 , ~ &(x,t)  \in \Omega^c\times(t_1,t_2] ,\\
  u(x,t)\geq 0 , ~ &(x,t)  \in \Omega\times(-\infty,t_1] ,
\end{array}
\right.
\end{equation}
then $u(x,t)\geq 0$ in $\Omega\times(t_1,t_2]$.

We can see that due to the nonlocal nature of the fractional time derivative $\partial_t^\alpha$ defined in \eqref{fratime}, in order to guarantee the validity of the classical maximum principle, one must prescribe the initial condition on the whole past time before $t_1$ instead just on the initial time moment $t_1$. If only requiring the initial condition  $u(x,t_1)\geq 0$ for \eqref{NMP}, then in general the conclusion of the maximum principle may not be valid as will be illustrated by the following counterexample.

For simplicity, we consider functions of $t$ only. Let
\begin{equation*}
u(x,t):=u(t):=\left\{
\begin{array}{ll}
    \sin t ,~   &t \in  (0,2\pi]  , \\
  t , ~ &t \in (-R,0] ,\\
  -R, ~ &t  \in (-\infty,-R].
\end{array}
\right.
\end{equation*}
Through a straightforward calculation, we have, for a sufficiently large $R>0$,
\begin{equation*}
\left\{
\begin{array}{ll}
    \partial_t^\alpha u(t)\geq 0 ,~   &\mbox{in} \,\, (0,2\pi] , \\
  u(0)=0.
\end{array}
\right.
\end{equation*}
However this differential inequality and the initial condition at time moment $t=0$ does not guarantee $u(t)$ to be nonnegative in $(0,2\pi]$,
while apparently, $u(t)<0$ in $(\pi, 2\pi)$. The main problem here lies in that $u(t) < 0$ for $t<0$.

This example shows that
the initial condition on the whole past time before $t_1$ is necessary to ensure the validity of
the maximum principle for parabolic problems with nonlocal time derivative.

Among the literature on qualitative properties of fractional parabolic equations with local time derivative, we would like to mention \cite{CWu2}, in which Chen and Wu considered the following
\begin{equation}\label{loctimemodel}
\left\{
\begin{array}{ll}
    \partial_t u(x,t)+(-\Delta)^su(x,t)=f(u(x,t)) ,~  &\mbox{in}\,\,  \mathbb{R}^n_+\times\mathbb{R}  , \\
  u(x,t)= 0 , ~ &\mbox{in}\,\, (\mathbb{R}^n  \backslash \mathbb{R}^n_+) \times\mathbb{R},
\end{array}
\right.
\end{equation}
They show that  the positive solution $u(x,t)$ is strictly increasing with respect
to $x_1$ in $\mathbb{R}^n_+$ for any $t\in\mathbb{R}$ by using the method of moving planes.

In the first step, under a weak assumption that the antisymmetric functions is allowed to tend to infinity, they established a narrow region principle and a maximum principle for antisymmetric functions, which are essential ingredients to carry on the direct method of moving planes. However, in the second step, when they employed the limit argument, they still need to assume that the solution $u(x,t)$ be bounded in $\mathbb{R}^n_+\times\mathbb{R}$.

Inspired by the previous literature, in this paper, we will investigate qualitative properties of solutions of dual-fractional parabolic problem \eqref{model}. Under notably weaker conditions than in \cite{CWu2}, that is, allowing the solutions to tend to infinity in some rate,
we will prove that they are strictly increasing with respect to $x_1$  in $\mathbb{R}^n_+$ for any $t\in\mathbb{R}$. To this aim, we need to introduce some new ideas and approaches to surmount the difficulties caused by both the dual non-locality in space-time and the weaker assumption on the solutions, as we will deliberate after the introduction of each theorem below.

To illustrate the main results of this paper, we start by presenting the notation that will be used in what follows. Let $$T_\lambda:=\{x=(x_1,x')\in \mathbb R^n \mid x_1=\lambda\,\,\mbox{for}\,\,\lambda \in \mathbb R\}$$
be the moving planes perpendicular to $x_1$-axis,
$$ \Sigma_\lambda:= \{x\in \mathbb R^n \mid x_1< \lambda \}\,\,\mbox{and} \,\,\Omega_\lambda:= \{x\in \mathbb R^n_+ \mid x_1< \lambda \}$$
be the region to the left of the hyperplane $T_\lambda$ in $\mathbb{R}^n$ and in $\mathbb{R}^n_+$ respectively. We denote the reflection of $x$ with respect to the hyperplane $T_\lambda$ as
$$x^\lambda:=(2\lambda-x_1, x_2,\cdots, x_n).$$
Let $u(x,t)$ be a solution of \eqref{model} and
$u_{\lambda}(x,t):=u(x^{\lambda},t)$. Define
\begin{equation*}
w_{\lambda} (x,t) :=u_{\lambda}(x,t) - u(x,t),
\end{equation*}
which represents the comparison between the values of $u(x,t)$ and $u(x^{\lambda},t)$. It is obvious that $w_{\lambda}(x,t)$ is an antisymmetric
function of $x$ with respect to the hyperplane $T_\lambda$.

We are now in a position to state our main results of this paper. The first one is the narrow region principle for antisymmetric functions in unbounded domains.
\begin{theorem}\label{NRP}{\rm(Narrow region principle for antisymmetric functions)}
Let $\Omega$ be an unbounded narrow region containing in the narrow slab $\{x\in\Sigma_\lambda \mid \lambda-2l<x_1<\lambda\}$  with some small $l$. Suppose that $$w(x,t)\in(C^{1,1}_{\rm loc}(\Omega) \cap {\mathcal L}_{2s}(\mathbb{R}^n)) \times (C^1(\mathbb{R})\cap {\mathcal L}^{-}_{\alpha}(\mathbb{R}))$$ is lower semi-continuous with respect to $x$ on $\overline{\Omega}$ , satisfying
\begin{equation}\label{NRP1}
  w(x,t)\geq -C(1+|x|^\gamma) \,\, \mbox{for some}\,\, 0<\gamma<2s,
\end{equation}
and
\begin{equation}\label{NRP2}
\left\{
\begin{array}{ll}
    \partial_t^\alpha w(x,t)+(-\Delta)^sw(x,t)=c(x,t) w(x,t) ,~   &(x,t) \in  \Omega\times\mathbb{R}  , \\
  w(x,t)\geq 0 , ~ &(x,t)  \in (\Sigma_\lambda  \backslash \Omega) \times\mathbb{R} , \\
  w(x,t)=- w(x^\lambda,t), &(x,t)     \in \Sigma_\lambda\times\mathbb{R}  ,
\end{array}
\right.
\end{equation}
where $c(x,t)$ is bounded from above.

Then
\begin{equation}\label{NRP3}
w(x,t)\geq0, \,\,\mbox{in} \,\, \Sigma_\lambda\times\mathbb{R}
\end{equation}
for sufficiently small $l$.
 Furthermore, if $w(x,t)$ attains zero at some point $(x^0,t_0)\in\Omega\times\mathbb{R}$, then
\begin{equation}\label{NRP3-1}
 w(x,t)\equiv0, \,\, \mbox{in}\,\, \mathbb{R}^n\times(-\infty,t_0].
\end{equation}
\end{theorem}

Note that the condition \eqref{NRP1} allows $w(x,t)$ to tend to negative infinity as $|x|\rightarrow\infty$ as long as its decreasing rate is not greater than $|x|^\gamma$.
The previous research has shown that when studying maximum principles in unbounded domains, one usually need to impose some decay (to zero) assumption on $w$ at infinity, or at least to assume that $w$ is bounded from below. In this respect, Chen and Wu \cite{CWu2} weakened the condition to the type of \eqref{NRP1} for the first time by constructing a suitable auxiliary function to establish the narrow region principle for the fractional parabolic equation involving the local time derivative $\partial_t$. However, the auxiliary function given in \cite{CWu2} has failed to work in our situation due to the presence of fractional time derivative $\partial_t^\alpha$. Here we introduce a new approach by applying a perturbation technique in $t$ to construct a sequence of auxiliary functions to estimate the dual-fractional operator $\partial_t^\alpha+(-\Delta)^s$ along a sequence of approximate minimum points.

The aforementioned narrow region principle is a crucial ingredient to carry out the method of moving planes as it provides a starting point, then
from which we will move the plane $T_\lambda$ along $x_1$ direction to the right as long as inequality (\ref{NRP3}) (with $w(x,t)$ replaced by
$w_\lambda (x,t)$) holds. Let
$$ \lambda_0 = \sup \{ \lambda \mid w_\mu (x,t) \geq 0, \, \forall \, (x,t) \in \Sigma_\lambda \times \mathbb{R}, \mu \leq \lambda \}. $$
We will show that $\lambda_0 = + \infty.$ This is usually done by a contradiction argument. Suppose otherwise, then there exists $\lambda_k \searrow \lambda_0$ such that $\inf w_{\lambda_k} < 0$. In \cite{CWu2}, the authors took limit along a proper subsequence of such $\{w_{\lambda_k}\}$, then applied a Hopf type lemma to the limiting equation to derive a contradiction. In this process, it is necessary to require the  solutions be bounded.

In order to remove the boundedness assumption on solutions in \cite{CWu2}, we introduce a new idea here. Instead of taking limit along a subsequence of $\{w_{\lambda_k}\}$, we apply the {\em averaging effects} (see the theorems below)  twice, first on the solution $u$, and then on $w_{\lambda_0}$, to derive a contradiction for sufficiently large $k$. We believe that this new approach will become a very useful tool in investigating an unbounded sequence of solutions.

\begin{theorem}\label{AE} {\rm(Averaging effects)}
Let $D$ be a domain in $\mathbb{R}^n$ and $t_0\in\mathbb{R}$ be a real number. For any $x^0\in \mathbb{R}^n$, if there exists a positive radius $r>0$ such that $B_r(x^0)\cap \overline{D}=\varnothing$ as shown in Figure 1 and
\begin{equation}\label{AE1-1}
  u(x,t)\geq C_0>0\,\,\mbox{ in}\,\, D\times(t_0-r^{\frac{2s}{\alpha}},t_0+r^{\frac{2s}{\alpha}}].
\end{equation}
Assume that $$u(x,t)\in\left(C^{1,1}_{\rm loc}(B_r(x^0)) \cap {\mathcal L}_{2s}(\mathbb{R}^n)\right) \times \left(C^1([t_0-r^{\frac{2s}{\alpha}},t_0+r^{\frac{2s}{\alpha}}])\cap {\mathcal L}^{-}_{\alpha}(\mathbb{R})\right)$$ is lower semi-continuous in $x$ on $\overline{B_r(x^0)}$, satisfying
\begin{equation}\label{AE1}
\left\{
\begin{array}{ll}
    \partial_t^\alpha u(x,t)+(-\Delta)^su(x,t)\geq-\varepsilon,~   &(x,t) \in  B_r(x^0)\times(t_0-r^{\frac{2s}{\alpha}},t_0+r^{\frac{2s}{\alpha}}] , \\
  u(x,t)\geq 0 , ~ &(x,t)  \in B^c_r(x^0)\times(t_0-r^{\frac{2s}{\alpha}},t_0+r^{\frac{2s}{\alpha}}] ,\\
  u(x,t)\geq 0 , ~ &(x,t)  \in B_r(x^0)\times(-\infty,t_0-r^{\frac{2s}{\alpha}}] ,
\end{array}
\right.
\end{equation}
for some sufficiently small $\varepsilon>0$. Then there exists a positive constant $$C_1=C_1\left(\alpha,n,s, C_0, \operatorname{diam}(D), \operatorname{dist}(x^0, D)\right)$$ such that
\begin{equation*}
  u(x^0,t_0)\geq C_1>0,
\end{equation*}
where $\operatorname{diam}(D)$ represents the diameter of domain $D$, and $\operatorname{dist}(x^0, D)$ denotes the distance from point $x^0$ to domain $D$.
\end{theorem}
 \begin{center}
\begin{tikzpicture}[node distance = 0.3cm]
\draw [thick](-1,1) ellipse  [x radius=1.5cm, y radius=0.8cm];
\fill[pink!50] (-1,1) ellipse  [x radius=1.5cm, y radius=0.8cm];
\path (-1,1.3) [blue][semithick] node [ font=\fontsize{10}{10}\selectfont] {$u\geqslant C_0>0$};
\path (-1,0.8) node  {$D$};
\draw (2.0,1) circle (0.8);
\path (2.0,1) [very thin,fill=black]  circle(1 pt) node at (2.15, 0.8) {$ x^0$};
\draw [thick] [dashed] [red] (2.0,1)--(2.7,1.4) node at (2.35, 1.35)  [red] {$r$};
\node [below=0.5cm, align=flush center,text width=16cm] at (0,-0.1)
        {Figure 1. The positional relationship between the region $D$ and the ball $B_r(x^0)$ in $\mathbb{R}^n$. };
\end{tikzpicture}
\end{center}
\begin{theorem}\label{AEA} {\rm(Averaging effects for antisymmetric functions)}
Let $D\subset\Sigma_\lambda$ be a domain and $t_0\in\mathbb{R}$ be a real number. For any $x^0\in \Sigma_\lambda$\,, if there exists a ball $B_r(x^0)\subset\Sigma_\lambda$ such that $B_r(x^0)\cap \overline{D}=\varnothing$ and $r\leq \frac{\operatorname{dist}(x^0,T_\lambda)}{2}$ as shown in Figure 2, and
\begin{equation}\label{AEA1}
  w(x,t)\geq C_0>0\,\,\mbox{ in}\,\, D\times(t_0-r^{\frac{2s}{\alpha}},t_0+r^{\frac{2s}{\alpha}}].
\end{equation}
Assume that $$w(x,t)\in\left(C^{1,1}_{\rm loc}(B_r(x^0)) \cap {\mathcal L}_{2s}(\mathbb{R}^n)\right) \times \left(C^1([t_0-r^{\frac{2s}{\alpha}},t_0+r^{\frac{2s}{\alpha}}])\cap {\mathcal L}^{-}_{\alpha}(\mathbb{R})\right)$$ is lower semi-continuous in $x$ on $\overline{B_r(x^0)}$, satisfying
\begin{equation}\label{AEA2}
\left\{
\begin{array}{ll}
    \partial_t^\alpha w(x,t)+(-\Delta)^sw(x,t)\geq -\varepsilon,~   &(x,t) \in  B_r(x^0)\times(t_0-r^{\frac{2s}{\alpha}},t_0+r^{\frac{2s}{\alpha}}] , \\
  w(x,t)\geq 0 , ~ &(x,t)  \in \left(\Sigma_\lambda\setminus B_r(x^0)\right)\times(t_0-r^{\frac{2s}{\alpha}},t_0+r^{\frac{2s}{\alpha}}] ,\\
  w(x,t)\geq 0 , ~ &(x,t)  \in B_r(x^0)\times(-\infty,t_0-r^{\frac{2s}{\alpha}}] ,\\
  w(x,t)=- w(x^\lambda,t), &(x,t)     \in \Sigma_\lambda\times\mathbb{R} ,
\end{array}
\right.
\end{equation}
for some sufficiently small $\varepsilon>0$. Then there exists a positive constant $$C_1=C_1\left(\alpha,n,s, C_0, \operatorname{diam}(D), \operatorname{dist}(x^0, D),\operatorname{dist}(\partial D,T_\lambda),\operatorname{dist}(x^0,T_\lambda)\right)$$ such that
\begin{equation*}
  w(x^0,t_0)\geq C_1>0,
\end{equation*}
where $\operatorname{dist}(\partial D,T_\lambda)$ stands for the distance between the boundary $\partial D$ and the hyperplane $T_\lambda$.
\end{theorem}
 \begin{center}
\begin{tikzpicture}[node distance = 0.3cm]
\draw (-2,1)[thick] ellipse  [x radius=1.5cm, y radius=0.8cm];
\fill[pink!50] (-2,1) ellipse  [x radius=1.5cm, y radius=0.8cm];
\path (-2,1.3) [blue][semithick] node [ font=\fontsize{10}{10}\selectfont] {$w\geqslant C_0>0$};
\path (-0.1,2.4) node [ font=\fontsize{12}{12}\selectfont] {$\Sigma_\lambda$};
\path (-2,0.8) node [ font=\fontsize{10}{10}\selectfont] {$D$};
\draw (1,1.2) circle (0.5);
\path (1,1.2) [very thin,fill=black]  circle(1 pt) node at (1.15, 1)[font=\fontsize{10}{10}\selectfont]  {$ x^0$};
\draw [thick] [dashed] [red] (1,1.2)--(1.5,1.3) node at (1.25, 1.4) [font=\fontsize{10}{10}\selectfont] [red] {$r$};
\draw  [->,semithick](-4,-0.2)--(4,-0.2) node [anchor=north west] {$x_1$};
\draw [semithick] (2,-1.1) -- (2,3);
\path (2.35,-0.75) node  [ font=\fontsize{10}{10}\selectfont] {$T_{\lambda}$};
\node [below=1cm, align=flush center,text width=16cm] at (0,-0.1)
        {Figure 2. The positional relationship between the region $D$ and the ball $B_r(x^0)$ in $\Sigma_\lambda$. };
\end{tikzpicture}
\end{center}
\begin{remark}
The so-called averaging effects reveal that the positiveness of the solution to the fractional time-diffusion equations in some region $D$ will be diffused to any other region $B$ disjointed from $D$. The averaging effects depend on the distance between such two regions, the closer the distance, the greater the effect.
\end{remark}

Based on the vital techniques above, we establish the direct method of moving planes suitable for the nonlocal parabolic equation \eqref{model} to derive the following monotonicity result.

\begin{theorem}\label{MainR}{\rm(Monotonicity in a half space)} Let
$$u(x,t)\in\left(C^{1,1}_{\rm loc}(\mathbb{R}^n_+) \cap {\mathcal L}_{2s}(\mathbb{R}^n)\right) \times \left(C^1(\mathbb{R})\cap {\mathcal L}^{-}_{\alpha}(\mathbb{R})\right)$$ be a positive solution of \eqref{model}.
Assume that $u(x,t)$ is uniformly continuous with respect to $x$, satisfying
\begin{equation*}
  u(x,t)\leq C(1+|x|^\gamma)\,\, \mbox{for some}\,\,0<\gamma<2s.
\end{equation*}
If the nonlinear term $f\in C^1([0,+\infty))$ satisfies $f(0)\geq 0$, $f'(0)\leq0$ and $f'$ is bounded from above, then the solution $u(x,t)$ is strictly increasing with respect to $x_1$ in $\mathbb{R}^n_+$ for any $t\in\mathbb{R}$.
\end{theorem}

\begin{remark}
To our knowledge, Theorem \ref{MainR} is the first result about the geometric shapes of solutions to the nonlocal parabolic equations involving the fractional time derivative $\partial_t^\alpha$ and the fractional Laplacian $(-\Delta)^s$. Particularly, the holistic approach presented here is also valid for parabolic equations of type \eqref{model} with the usual local time derivative $\partial_t$. It is worth mentioning that our results are still novel even if $\partial_t^\alpha$ is replaced by  $\partial_t$, since we remove the assumption that $u$ is uniformly bounded in $x$ as compared to the result in \cite{CWu2}\,.
\end{remark}

The remainder part of this paper proceeds as follows. Section \ref{2} is composed by the proofs of various maximum principles used for this study, including the aforementioned narrow region principle. Section \ref{3} is devoted to establishing the averaging effects applicable to the dual fractional operator $\partial_t^\alpha+(-\Delta)^s$. In section \ref{4}\,, we complete the proof of Theorem \ref{MainR}\,.  The last section presents some useful lemmas.

\section{Maximum principles }\label{2}
Without assuming any decay and boundedness conditions on the antisymmetric function $w(x,t)$ with respect to the space variable $x$, this section begins by establishing the narrow region principle ( Theorem \ref{NRP} ) in unbounded domains.
From now on,
$C$ denotes a constant whose value may be different from line to
line, and only the relevant dependence is specified in what follows.

\begin{proof}
[\bf Proof of Theorem \ref{NRP}\,.] \,

We argue by contradiction to derive \eqref{NRP3}.

Note that condition \eqref{NRP1} allows $w(x,t)$ to go to negative infinity as $|x|\rightarrow\infty$;
hence, the minimizing sequence of $w(x,t)$ in $x$ may leak to infinity and the minimum may not be attained.
To overcome this impediment, we employ the following positive auxiliary function
$$h(x):=\left[\left(1-\frac{(x_1-(\lambda-l))^2}{l^2}\right)_{+}^s+1\right](1+|x'|^2)^{\frac{\beta}{2}}$$
for some $\gamma<\beta<2s$
to ensure that
\begin{equation}\label{NRP-l}
  \displaystyle\lim_{|x|\rightarrow\infty}\bar{w}(x,t)
:=\displaystyle\lim_{|x|\rightarrow\infty}\frac{w(x,t)}{h(x)}\geq0.
\end{equation}
It follows that for each fixed $t\in\mathbb{R}$, if there is a point $x\in\Omega$ with $\bar{w}(x,t)<0$, then there must exists $x(t)\in \Omega$ such that
\begin{equation}\label{NRP4}
  \bar{w}(x(t),t)=\displaystyle\min_{x\in \Omega}\bar{w}(x,t)<0.
\end{equation}
Combining \eqref{NRP1} with the definition of $\bar{w}(x,t)$ and $\gamma<\beta$,
one deduces that $\bar{w}(x(t),t)$ is bounded. Hence, if \eqref{NRP3} is not valid, then there exists a positive constant $m$ such that
\begin{equation}\label{NRP5}
  \inf_{\Omega\times\mathbb{R}}\bar{w}(x,t)=\inf_{\mathbb{R}}\bar{w}(x(t),t)=:-m<0.
\end{equation}
This implies  that there exists a sequence $\{(x^k,t_k)\}\subset\Omega\times\mathbb{R}$ such that
\begin{equation*}
  \bar{w}(x^k,t_k)=-m_k\rightarrow -m \,\,\mbox{as}\,\, k\rightarrow\infty.
\end{equation*}
Let $\varepsilon_k:=m-m_k$, then it is obvious that $\varepsilon_k\geq0$ and $\varepsilon_k\rightarrow 0$ as $k\rightarrow0$. We further need to introduce the following auxiliary function
\begin{equation*}
  v_k(x,t):=\bar{w}(x,t)-\varepsilon_k\eta_k(t)
\end{equation*}
to remedy scenario that the infimum of $\bar{w}(x(t),t)$ may not be attained due to $t\in\mathbb{R}$,
where
$$\eta_k(t)=\eta(t-t_k)\in C_0^\infty\left((-2+t_k,2+t_k)\right)$$
is a smooth cut-off function satisfying
$$\eta_k(t)\equiv1\,\, \mbox{in}\,\, [-1+t_k,1+t_k],\,\,\mbox{and}\,\, 0\leq\eta_k(t)\leq 1.$$
One one hand, we have
\begin{equation*}
  v_k(x^k,t_k)=\bar{w}(x^k,t_k)-\varepsilon_k=-m_k-m+m_k=-m.
\end{equation*}
On the other hand, if $|t-t_k|\geq2$ and $x\in\Omega$, then it follows from \eqref{NRP5} that
\begin{equation*}
  v_k(x,t)=\bar{w}(x,t)\geq-m.
\end{equation*}
Thus, there exists $(\bar{x}^k,\bar{t}_k)\in\Omega\times(-2+t_k,2+t_k)$ such that
\begin{equation}\label{NRP6}
 -m-\varepsilon_k\leq v_k(\bar{x}^k,\bar{t}_k)= \inf_{\Omega\times\mathbb{R}}v_k(x,t)\leq-m.
\end{equation}
From this, it is not difficult to see that
\begin{equation}\label{NRP6-1}
  -m\leq\bar{w}(\bar{x}^k,\bar{t}_k)\leq-m+\varepsilon_k=-m_k<0.
\end{equation}

Next, by a direct calculation, we obtain
\begin{equation*}
  \partial_t^\alpha v_k(\bar{x}^k,\bar{t}_k)=C_\alpha\int_{-\infty}^{\bar{t}_k}\frac{v_k(\bar{x}^k,\bar{t}_k)-v_k(\bar{x}^k,\tau)}
  {(\bar{t}_k-\tau)^{1+\alpha}}\operatorname{d}\!\tau\leq0.
\end{equation*}
Then in terms of the definition of $v_k(x,t)$ and Lemma \ref{mlem1}\,, we derive
\begin{equation}\label{NRP7}
 \partial_t^\alpha \bar{w}(\bar{x}^k,\bar{t}_k)\leq \varepsilon_k\partial_t^\alpha \eta_k(\bar{t}_k)\leq C\varepsilon_k.
\end{equation}
To proceed, a combination of the definition of $\bar{w}(x,t)$, \eqref{NRP2}, \eqref{NRP4}, \eqref{NRP6-1} with $|\bar{x}^k-y|<|\bar{x}^k-y^\lambda|$ and $h(y)> h(y^\lambda)$ for $y\in\Sigma_\lambda$
yields that
\begin{eqnarray}\label{NRP8}
(-\Delta)^sw(\bar{x}^k,\bar{t}_k)&=&(-\Delta)^s\left(\bar{w}(\bar{x}^k,\bar{t}_k)h(\bar{x}^k)\right) \nonumber\\
&=& \bar{w}(\bar{x}^k,\bar{t}_k)(-\Delta)^sh(\bar{x}^k)+C_{n,s}P.V.\int_{\mathbb{R}^n}
\frac{h(y)\left( \bar{w}(\bar{x}^k,\bar{t}_k)- \bar{w}(y,\bar{t}_k)\right)}{|\bar{x}^k-y|^{n+2s}}\operatorname{d}\!y\nonumber\\
&\leq&\bar{w}(\bar{x}^k,\bar{t}_k)(-\Delta)^sh(\bar{x}^k)+C_{n,s}\int_{\Sigma_\lambda}
\frac{h(y)\bar{w}(\bar{x}^k,\bar{t}_k)- w(y,\bar{t}_k)}{|\bar{x}^k-y^\lambda|^{n+2s}}\operatorname{d}\!y\nonumber\\
&&+C_{n,s}\int_{\Sigma_\lambda}
\frac{h(y^\lambda)\bar{w}(\bar{x}^k,\bar{t}_k)+ w(y,\bar{t}_k)}{|\bar{x}^k-y^\lambda|^{n+2s}}\operatorname{d}\!y\nonumber\\
&\leq&\bar{w}(\bar{x}^k,\bar{t}_k)(-\Delta)^sh(\bar{x}^k)+C_{n,s}\int_{\Sigma_\lambda}
\frac{2h(y^\lambda)\bar{w}(\bar{x}^k,\bar{t}_k)}{|\bar{x}^k-y^\lambda|^{n+2s}}\operatorname{d}\!y\nonumber\\
&\leq&\bar{w}(\bar{x}^k,\bar{t}_k)(-\Delta)^sh(\bar{x}^k)\nonumber\\
&\leq&\frac{C_1}{l^{2s}}h(\bar{x}^k)\bar{w}(\bar{x}^k,\bar{t}_k),
\end{eqnarray}
where on the last line we used the fact established in \cite[Lemma 2.1]{CWu2} that there exits a positive constant $C_1$ such that
\begin{equation*}
  \frac{(-\Delta)^sh(x)}{h(x)}\geq\frac{C_1}{l^{2s}} \,\,\mbox{for all}\,\, \lambda-2l<x_1<\lambda \,\,\mbox{with sufficiently small}\,\, l.
\end{equation*}
Substituting \eqref{NRP7} and \eqref{NRP8} into \eqref{NRP2}, and combining \eqref{NRP6-1} with  the bounded-ness of $c(x,t)$ from above, we obtain
\begin{eqnarray*}
-\frac{C_1}{l^{2s}}m_kh(\bar{x}^k) &\geq& \frac{C_1}{l^{2s}}h(\bar{x}^k)\bar{w}(\bar{x}^k,\bar{t}_k)\\
 &\geq& (-\Delta)^sw(\bar{x}^k,\bar{t}_k) \\
   &=& - \partial_t^\alpha w(\bar{x}^k,\bar{t}_k)+c(\bar{x}^k,\bar{t}_k)w(\bar{x}^k,\bar{t}_k)\\
   &=& - h(\bar{x}^k)\partial_t^\alpha \bar{w}(\bar{x}^k,\bar{t}_k)+c(\bar{x}^k,\bar{t}_k)h(\bar{x}^k)\bar{w}(\bar{x}^k,\bar{t}_k)\\
   &\geq&-C\varepsilon_kh(\bar{x}^k)-Cmh(\bar{x}^k).
\end{eqnarray*}
Now dividing both side of the preceding inequality by $-mh(\bar{x}^k)$, we deduce that
\begin{equation*}
  \frac{C_1}{l^{2s}}\leftarrow\frac{C_1}{l^{2s}}\frac{m_k}{m}\leq\frac{C\varepsilon_k}{m}+ C\rightarrow C,
\end{equation*}
as $k\rightarrow\infty$.
Therefore, we derive a contradiction for sufficiently small $l$,
which concludes that \eqref{NRP3} is true.

Finally, we prove \eqref{NRP3-1}. It follows from \eqref{NRP3} that
$$w(x^0,t_0)=\min_{\Sigma_\lambda\times\mathbb{R}}w(x,t)=0.$$

If $w(x,t_0)\not\equiv0$ in $\Sigma_\lambda$, on one hand, we compute
\begin{equation*}
  \partial_t^\alpha w(x^0,t_0)=-C_\alpha\int_{-\infty}^{t_0}\frac{w(x^0,\tau)}
  {(t_0-\tau)^{1+\alpha}}\operatorname{d}\!\tau\leq0
\end{equation*}
and
\begin{equation*}
  (-\Delta)^sw(x^0,t_0)=C_{n,s} P.V. \int_{\Sigma_{\lambda}} w(y,t_0)\left(\frac{1}{|x^0-y^{\lambda}|^{n+2s}}-\frac{1}{|x^0-y|^{n+2s}}\right)\operatorname{d}\!y<0,
\end{equation*}
then
$$\partial_t^\alpha w(x^0,t_0)+(-\Delta)^sw(x^0,t_0)<0.$$
On the other hand, by virtue of the equation in \eqref{NRP2}, we have
\begin{equation}\label{NRP9}
  \partial_t^\alpha w(x^0,t_0)+(-\Delta)^sw(x^0,t_0)=0.
\end{equation}
This contradiction implies that $w(x,t_0)\equiv0$ in $\Sigma_\lambda$.
In addition, the antisymmetry of $w(x,t)$ with respect to $x$ infers that
$$w(x,t_0)\equiv0 \,\, \mbox{in} \,\,\mathbb{R}^n.$$

Thereby for any fixed $x\in\Sigma_\lambda$ such that $w(x,t)\not\equiv0$ in $(-\infty,t_0)$, by estimating similarly as above, we have
\begin{equation*}
  \partial_t^\alpha w(x,t_0)=-C_\alpha\int_{-\infty}^{t_0}\frac{w(x,\tau)}
  {(t_0-\tau)^{1+\alpha}}\operatorname{d}\!\tau<0
\end{equation*}
and
\begin{equation*}
  (-\Delta)^sw(x,t_0)=0,
\end{equation*}
 which also means that
$$\partial_t^\alpha w(x,t_0)+(-\Delta)^sw(x,t_0)<0.$$
It contradicts the equation \eqref{NRP9}, then $w(x,t)\equiv0$ in $\Sigma_\lambda\times(-\infty,t_0]$. Applying the antisymmetry of $w(x,t)$ with respect to $x$ again, we deduce that
$$w(x,t)\equiv0 \,\, \mbox{in}\,\, \mathbb{R}^n\times(-\infty,t_0].$$
Therefore, the proof of Theorem
\ref{NRP} is completed.
\end{proof}
In fact, if the region $\Omega$ given in Theorem \ref{NRP} is not narrow, then it suffices to assume that the positive part of coefficient $c(x,t)$
is small to
guarantee the validity of the similar maximum principle, as pointed out in the following result.
\begin{theorem}\label{MPAF}{\rm(Maximum principle for antisymmetric functions)}
Let $\Omega$ be an unbounded domain in $\Sigma_\lambda$ with a finite width in $x_1$ direction. Suppose that $$w(x,t)\in(C^{1,1}_{\rm loc}(\Omega) \cap {\mathcal L}_{2s}(\mathbb{R}^n)) \times (C^1(\mathbb{R})\cap {\mathcal L}^{-}_{\alpha}(\mathbb{R}))$$ is lower semi-continuous for $x$ on $\overline{\Omega}$, satisfying
\begin{equation}\label{MPAF1}
   w(x,t)\geq -C(1+|x|^\gamma) \,\, \mbox{for some}\,\, 0<\gamma<2s,
\end{equation}
and
\begin{equation}\label{MPAF2}
\left\{
\begin{array}{ll}
    \partial_t^\alpha w(x,t)+(-\Delta)^sw(x,t)=c(x,t) w(x,t) ,~   &(x,t) \in  \Omega\times\mathbb{R}  , \\
  w(x,t)\geq 0 , ~ &(x,t)  \in (\Sigma_\lambda  \backslash \Omega) \times\mathbb{R} , \\
  w(x,t)=- w(x^\lambda,t), &(x,t)     \in \Sigma_\lambda\times\mathbb{R}  ,
\end{array}
\right.
\end{equation}
where $c(x,t)\leq c_0$ in $\Omega\times\mathbb{R}$ for some small positive constant $c_0$.

Then there holds that
\begin{equation}\label{MPAF3}
w(x,t)\geq0 \,\,\mbox{in} \,\,  \Sigma_\lambda\times\mathbb{R}.
\end{equation}
Furthermore, if $w(x^0,t_0)=0$  for some point $(x^0,t_0)\in\Omega\times\mathbb{R}$, then
\begin{equation*}
  w(x,t)\equiv0 \,\, \mbox{in}\,\, \mathbb{R}^n\times(-\infty,t_0].
\end{equation*}
\end{theorem}
\begin{proof}
[\bf Proof.] Since the width of the domain $\Omega$ in $x_1$ direction is finite, then it may well be assumed that $\Omega$ is contained in $\{x\in \Sigma_\lambda\mid \lambda-2a< x_1<\lambda\}$ for some $a>0$.
Thereby all proofs follow almost verbatim from the proof of Theorem \ref{NRP}\,, the essential  difference is that we choose
the auxiliary function
$$h(x):=\left[\left(1-\frac{(x_1-(\lambda-a))^2}{a^2}\right)_{+}^s+1\right](1+|bx'|^2)^{\frac{\beta}{2}}$$
for some $\gamma<\beta<2s$, where the sufficiently small positive constant $b$ depends on $a$ such that
$$\frac{(-\Delta)^sh(x)}{h(x)}\geq\frac{C_1}{a^{2s}}$$ for any $x\in \Omega$ and some constant $C_1>0$. In analogy with the notation and calculations in the proof of Theorem \ref{NRP}\,, if \eqref{MPAF3} is false, then we can finally deduce that
\begin{equation*}
  \frac{C_1}{a^{2s}}\leftarrow\frac{C_1}{a^{2s}}\frac{m_k}{m}\leq\frac{C\varepsilon_k}{m}+ c_0\rightarrow c_0,
\end{equation*}
as $k\rightarrow\infty$. It suffices to select the upper bound $c_0$ of the coefficient $c(x,t)$ to satisfy $c_0<\frac{C_1}{a^{2s}}$, then we derive a contradiction. Hence, \eqref{MPAF3} is valid.
Furthermore, by proceeding similarly as the proof of Theorem \ref{NRP}\,, we can verify that the strong maximum principle holds. This completes the proof of Theorem \ref{MPAF}.
\end{proof}

In the sequel, we will also use the following two simple maximum principles in bounded domains, where the minimum can be attained as compared to unbounded domains.
\begin{theorem}\label{MP} {\rm(Maximum principle in bounded domains)}
Let $\Omega$ be a bounded domain in $\mathbb{R}^n$ and $t_1<t_2$ be two real numbers.
Suppose that $$u(x,t)\in(C^{1,1}_{\rm loc}(\Omega) \cap {\mathcal L}_{2s}(\mathbb{R}^n)) \times (C^1([t_1,t_2])\cap {\mathcal L}^{-}_{\alpha}(\mathbb{R}))$$ is lower semi-continuous in $x$ on $\overline{\Omega}$, satisfying
\begin{equation}\label{MP1}
\left\{
\begin{array}{ll}
    \partial_t^\alpha u(x,t)+(-\Delta)^su(x,t)\geq 0 ,~   &(x,t) \in  \Omega\times(t_1,t_2]  , \\
  u(x,t)\geq 0 , ~ &(x,t)  \in \Omega^c\times(t_1,t_2] ,\\
  u(x,t)\geq 0 , ~ &(x,t)  \in \Omega\times(-\infty,t_1] ,
\end{array}
\right.
\end{equation}
then $u(x,t)\geq 0$ in $\Omega\times(t_1,t_2]$.
\end{theorem}
\begin{proof}[\bf Proof.]
The proof goes by contradiction. If the conclusion is not valid, then
there exists $(x^0,t_0)\in \Omega\times(t_1,t_2]$ such that
$$u(x^0,t_0)=\min_{\Omega\times (t_1,t_2]}u(x,t)<0.$$
Combining the exterior condition with the initial condition in \eqref{MP1}, we directly calculate
\begin{eqnarray*}
    &&\partial_t^\alpha u(x^0,t_0)+(-\Delta)^su(x^0,t_0) \\
   &=&C_{\alpha}\int_{-\infty}^{t_0} \frac{u(x^0,t_0)-u(x^0,\tau)}{(t_0-\tau)^{1+\alpha}}\operatorname{d}\!\tau+C_{n,s} P.V. \int_{\mathbb R^n} \frac{u(x^0,t_0)-u(y,t_0)}{|x^0-y|^{n+2s}}\operatorname{d}\!y\\
    &=&C_{\alpha}\int_{-\infty}^{t_1} \frac{u(x^0,t_0)-u(x^0,\tau)}{(t_0-\tau)^{1+\alpha}}\operatorname{d}\!\tau+C_{\alpha}\int_{t_1}^{t_0} \frac{u(x^0,t_0)-u(x^0,\tau)}{(t_0-\tau)^{1+\alpha}}\operatorname{d}\!\tau\\
    &&+C_{n,s} P.V. \int_{\Omega} \frac{u(x^0,t_0)-u(y,t_0)}{|x^0-y|^{n+2s}}\operatorname{d}\!y+C_{n,s} \int_{\Omega^c} \frac{u(x^0,t_0)-u(y,t_0)}{|x^0-y|^{n+2s}}\operatorname{d}\!y\\
    &<&0,
\end{eqnarray*}
which contradicts the differential inequality in \eqref{MP1}. Hence, we verify Theorem \ref{MP}.
\end{proof}

\begin{theorem}\label{MPA} {\rm(Maximum principle for antisymmetric functions in bounded domains)} Let $\Omega\subset\Sigma_\lambda$ be a bounded domain and $[t_1,t_2]\subset\mathbb{R}$ be a finite interval.
Suppose that $$w(x,t)\in(C^{1,1}_{\rm loc}(\Omega) \cap {\mathcal L}_{2s}(\mathbb{R}^n)) \times (C^1([t_1,t_2])\cap {\mathcal L}^{-}_{\alpha}(\mathbb{R}))$$ is lower semi-continuous in $x$ on $\overline{\Omega}$, satisfying
\begin{equation}\label{MPA1}
\left\{
\begin{array}{ll}
    \partial_t^\alpha w(x,t)+(-\Delta)^sw(x,t)\geq 0 ,~   &(x,t) \in  \Omega\times(t_1,t_2]  , \\
  w(x,t)\geq 0 , ~ &(x,t)  \in (\Sigma_{\lambda}\setminus\Omega)\times(t_1,t_2] ,\\
  w(x,t)\geq 0 , ~ &(x,t)  \in \Omega\times(-\infty,t_1] ,\\
   w(x,t)=- w(x^\lambda,t), &(x,t)     \in \Sigma_\lambda\times\mathbb{R} ,
\end{array}
\right.
\end{equation}
then $w(x,t)\geq 0$ in $\Omega\times(t_1,t_2]$.
\end{theorem}
\begin{proof}[\bf Proof.]
We argue by contradiction. If the conclusion is violated, then there exists $(x^0,t_0)\in \Omega\times(t_1,t_2]$ such that
$$w(x^0,t_0)=\min_{\Sigma_\lambda\times (t_1,t_2]}w(x,t)<0.$$
Now a combination of the exterior condition, the initial condition with the antisymmetry of $w(x,t)$ in $x$ yields that
\begin{eqnarray*}
    &&\partial_t^\alpha w(x^0,t_0)+(-\Delta)^sw(x^0,t_0) \\
   &=&C_{\alpha}\int_{-\infty}^{t_0} \frac{w(x^0,t_0)-w(x^0,\tau)}{(t_0-\tau)^{1+\alpha}}\operatorname{d}\!\tau+C_{n,s} P.V. \int_{\mathbb R^n} \frac{w(x^0,t_0)-w(y,t_0)}{|x^0-y|^{n+2s}}\operatorname{d}\!y\\
    &=&C_{\alpha}\int_{-\infty}^{t_1} \frac{w(x^0,t_0)-w(x^0,\tau)}{(t_0-\tau)^{1+\alpha}}\operatorname{d}\!\tau+C_{\alpha}\int_{t_1}^{t_0} \frac{w(x^0,t_0)-w(x^0,\tau)}{(t_0-\tau)^{1+\alpha}}\operatorname{d}\!\tau\\
    &&+C_{n,s} P.V. \int_{\Sigma_\lambda} \frac{w(x^0,t_0)-w(y,t_0)}{|x^0-y|^{n+2s}}\operatorname{d}\!y+C_{n,s} \int_{\Sigma_\lambda^c} \frac{w(x^0,t_0)-w(y,t_0)}{|x^0-y|^{n+2s}}\operatorname{d}\!y\\
    &<&C_{n,s} \int_{\Sigma_\lambda} \frac{w(x^0,t_0)-w(y,t_0)}{|x^0-y^\lambda|^{n+2s}}\operatorname{d}\!y+C_{n,s} \int_{\Sigma_\lambda} \frac{w(x^0,t_0)+w(y,t_0)}{|x^0-y^\lambda|^{n+2s}}\operatorname{d}\!y\\
    &=&2C_{n,s} w(x^0,t_0)\int_{\Sigma_\lambda} \frac{1}{|x^0-y^\lambda|^{n+2s}}\operatorname{d}\!y<0,
\end{eqnarray*}
which is guaranteed by the radial decrease of the kernel resulting from $|x^0-y|<|x^0-y^\lambda|$ for $y\in \Sigma_\lambda$. Obviously, the aforementioned estimate contradicts the differential inequality in \eqref{MPA1}. This completes the proof of Theorem \ref{MPA}.
\end{proof}

\section{Averaging effects}\label{3}
In this section, we develop two averaging effects ( Theorem \ref{AE} and Theorem \ref{AEA} ) of the dual nonlocal operator $\partial_t^\alpha+(-\Delta)^s$, which
 will play crucial roles in the proof of our main results and will
become powerful tools in studying fractional problems, particularly in the cases where the solutions are unbounded. We start by presenting the proof of Theorem \ref{AE}\,.

\begin{proof}[\bf Proof of Theorem \ref{AE}\,.]
The lower bound estimate will be obtained by constructing a sub-solution. Denote
\begin{equation*}
\psi(x,t) := \phi(x)\eta(t):= C\left(1-\left|\frac{x-x^0}{r}\right|^{2}\right)_{+}^{s}\eta(t),
\end{equation*}
where $\eta(t)$ is a smooth cut-off function whose compact support  is contained in $(t_0-r^{\frac{2s}{\alpha}},t_0+r^{\frac{2s}{\alpha}})$, satisfying
$$0\leq\eta(t)\leq1, \,\,\mbox{and}\,\,\eta(t)\equiv1\,\, \mbox{in}\,\, [t_0-\frac{r^{\frac{2s}{\alpha}}}{2},t_0+\frac{r^{\frac{2s}{\alpha}}}{2}].$$
In fact, it is well known that
\begin{equation}\label{AE2}
(-\Delta)^s\phi(x)=\frac{1}{r^{2s}}\,\,\mbox{in}\,\, B_r(x^0)
\end{equation}
for a proper choice of positive constant $C$, and $\phi(x)\equiv 0$ in $B^c_r(x^0)$.
Let
$$\underline{u}(x,t):=u(x,t)\chi_D(x)+\delta
\psi(x,t),$$
where $\delta$ is a positive constant to be determined later, and $\chi_D$ is the characteristic function on $D$, more precisely, $\chi_D(x)\equiv1$ for $x\in D$ and $\chi_D(x)\equiv0$ for $x\notin D$.

Next, we attempt to claim that $\underline{u}(x,t)$ is
a sub-solution of $u(x,t)$ in $B_r(x^0)\times(t_0-r^{\frac{2s}{\alpha}},t_0+r^{\frac{2s}{\alpha}}]$.
For $(x,t)\in B_r(x^0)\times(t_0-r^{\frac{2s}{\alpha}},t_0+r^{\frac{2s}{\alpha}}]$, combining \eqref{AE1-1}, \eqref{AE1}, \eqref{AE2} and Corollary \ref{coro1} with $B_r(x^0)\cap \overline{D}=\varnothing$,
we directly calculate
\begin{eqnarray*}
  && \partial_t^\alpha \left(u(x,t)-\underline{u}(x,t)\right)+(-\Delta)^s\left(u(x,t)-\underline{u}(x,t)\right) \\
   &\geq& -\varepsilon -\delta\phi(x)\partial_t^\alpha \eta(t)-(-\Delta)^s\left(u(x,t)\chi_D(x)\right)-\delta\eta(t)(-\Delta)^s\phi(x)\\
   &\geq&-\varepsilon+C_{n,s}C_0\int_{D}\frac{1}{|x-y|^{n+2s}}\operatorname{d}\!y-\frac{C\delta}{r^{2s}}\\
   &\geq&-\varepsilon+C_2-\frac{C\delta}{r^{2s}}.
\end{eqnarray*}
Now choosing $\epsilon=\frac{C_2}{2}$ and $\delta =\frac{C_2r^{2s}}{2C}$, we deduce that
\begin{equation*}
 \partial_t^\alpha \left(u(x,t)-\underline{u}(x,t)\right)+(-\Delta)^s\left(u(x,t)-\underline{u}(x,t)\right)\geq0,\quad (x,t)\in B_r(x^0)\times(t_0-r^{\frac{2s}{\alpha}},t_0+r^{\frac{2s}{\alpha}}].
\end{equation*}
For $(x,t)\in  B^c_r(x^0)\times(t_0-r^{\frac{2s}{\alpha}},t_0+r^{\frac{2s}{\alpha}}]$, we have
$$u(x,t)-\underline{u}(x,t)=u(x,t)-u(x,t)\chi_D(x)\geq0,$$
which is ensured by the definition of $\psi(x,t)$ and the exterior condition in \eqref{AE1}.
While if $(x,t)\in  B_r(x^0)\times(-\infty,t_0-r^{\frac{2s}{\alpha}}]$, then we apply the initial condition in \eqref{AE1} to derive
$$u(x,t)-\underline{u}(x,t)=u(x,t)\geq0.$$
In summary, we have obtained
\begin{equation*}
\begin{cases}
\partial_t^\alpha \left(u(x,t)-\underline{u}(x,t)\right)+(-\Delta)^s\left(u(x,t)-\underline{u}(x,t)\right)\geq0, & \mbox{in} \,\, B_r(x^0)\times(t_0-r^{\frac{2s}{\alpha}},t_0+r^{\frac{2s}{\alpha}}], \\
u(x,t)-\underline{u}(x,t)\geq 0, & \mbox{in} \,\, B^c_r(x^0)\times(t_0-r^{\frac{2s}{\alpha}},t_0+r^{\frac{2s}{\alpha}}],\\
u(x,t)-\underline{u}(x,t)\geq 0, & \mbox{in} \,\, B_r(x^0)\times(-\infty,t_0-r^{\frac{2s}{\alpha}}].\\
\end{cases}
\end{equation*}
Then the maximum principle established in Theorem \ref{MP} implies that
\begin{equation*}
u(x,t)\geq\underline{u}(x,t) \,\,\mbox{for}\,\, (x,t)\in B_r(x^0)\times(t_0-r^{\frac{2s}{\alpha}},t_0+r^{\frac{2s}{\alpha}}].
\end{equation*}
Hence, we verify that $\underline{u}(x,t)$ is
a sub-solution of $u(x,t)$ in $B_r(x^0)\times(t_0-r^{\frac{2s}{\alpha}},t_0+r^{\frac{2s}{\alpha}})$. Finally,
it follows that
\begin{equation*}
 u(x^0,t_0)\geq \underline{u}(x^0,t_0)=\delta
\phi(x^0)\eta(t_0)=C\delta=:C_1>0.
\end{equation*}
Now the proof of Theorem \ref{AE} is completed.
\end{proof}

Now we turn our attention to the proof of Theorem \ref{AEA} regarding the averaging effects for the antisymmetric functions.

\begin{proof}[\bf Proof of Theorem \ref{AEA}\,.] The primary objective of this proof is to
construct an antisymmetric sub-solution of $w(x,t)$. Let
\begin{equation*}
 \phi(x):=\left(1-\left|\frac{x-x^0}{r}\right|^{2}\right)_{+}^{s} \,\,\mbox{and}\,\, \phi_\lambda(x):=\left(1-\left|\frac{x^\lambda-x^0}{r}\right|^{2}\right)_{+}^{s},
\end{equation*}
then it is obvious that
\begin{equation*}
  \Phi(x):=\phi(x)-\phi_\lambda(x)
\end{equation*}
is an antisymmetric function with respect to the plane $T_\lambda$.
We denote
$$\eta(t)\in C_0^\infty((t_0-r^{\frac{2s}{\alpha}},t_0+r^{\frac{2s}{\alpha}}))$$ which is a smooth cut-off function whose value belongs to $[0,1]$, satisfying
$$\eta(t)\equiv1\,\,\mbox{in}\,\, [t_0-\frac{r^{\frac{2s}{\alpha}}}{2},t_0+\frac{r^{\frac{2s}{\alpha}}}{2}].$$
Let $D^\lambda$ be the reflection domain of $D$ with respect to the plane $T_{\lambda}$ as illustrated in Figure 3 below, and
$$\underline{w}(x,t)=w(x,t)\chi_{D\cup D^\lambda}(x)+\delta\Phi(x)\eta(t),$$
where $\delta$ is a positive constant to be determined later,
and $\chi_{D\cup D^\lambda}(x)$ is the characteristic function in the region $D\cup D^\lambda$.

\begin{center}
\begin{tikzpicture}[node distance = 0.3cm]
\draw (-2,1)[thick] ellipse  [x radius=1.5cm, y radius=0.8cm];
\fill[pink!50] (-2,1) ellipse  [x radius=1.5cm, y radius=0.8cm];
\draw (2,1)[thick] ellipse  [x radius=1.5cm, y radius=0.8cm];
\fill[pink!50] (2,1) ellipse  [x radius=1.5cm, y radius=0.8cm];
\path (-2,1.3) [blue][semithick] node [ font=\fontsize{10}{10}\selectfont] {$w\geqslant C_0>0$};
\path (-3,2.4) node [ font=\fontsize{12}{12}\selectfont] {$\Sigma_\lambda$};
\path (-2,0.8) node {$D$};
\path (2,0.8) node {$D^\lambda$};
\draw (-1.1,2.5) circle (0.5);
\path (-1.1,2.5) [very thin,fill=black]  circle(1 pt) node at (-0.9, 2.35) [font=\fontsize{10}{10}\selectfont]{$ x^0$};
\draw (1.1,2.5) circle (0.5);
\path (1.1,2.5) [very thin,fill=black]  circle(1 pt) node at (1.15, 2.25) [font=\fontsize{9}{9}\selectfont]{$ (x^0)^\lambda$};
\draw [thick] [dashed] [red] (-1.1,2.5)--(-1.5,2.2) node at (-1.3, 2.5) [font=\fontsize{10}{10}\selectfont] [red] {$r$};
\draw [thick] [dashed] [red] (1.1,2.5)--(1.6,2.7) node at (1.3, 2.7) [font=\fontsize{10}{10}\selectfont] [red] {$r$};
\draw  [->,semithick](-4,-0.2)--(4,-0.2) node [anchor=north west] {$x_1$};
\draw [semithick] (0,-1.1) -- (0,3.5);
\path (0.25,-0.75) node  [ font=\fontsize{10}{10}\selectfont] {$T_{\lambda}$};
\node [below=1cm, align=flush center,text width=16cm] at (0,-0.5)
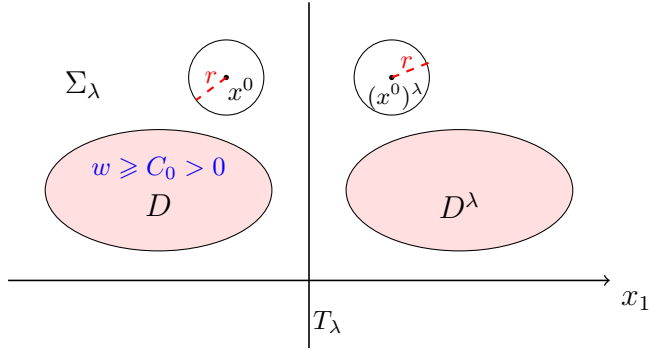
        {Figure 3. The positions of $B_r(x^0)$, $B_r((x^0)^{\lambda})$, $D$ and $D^\lambda$. };
\end{tikzpicture}
\end{center}

In the sequel, we aim to show that the antisymmetric function $\underline{w}(x,t)$ is
a sub-solution of $w(x,t)$ in $B_r(x^0)\times(t_0-r^{\frac{2s}{\alpha}},t_0+r^{\frac{2s}{\alpha}}]$.
For $(x,t)\in B_r(x^0)\times(t_0-r^{\frac{2s}{\alpha}},t_0+r^{\frac{2s}{\alpha}}]$, combining \eqref{AEA1} with \eqref{AEA2}, \eqref{AE2}, Corollary \ref{coro1} and $r\leq \frac{\operatorname{dist}(x^0,T_\lambda)}{2}$, we further select $\epsilon=\frac{C_2}{2}$ and $\delta =\frac{C_2r^{2s}}{2C}$ to derive
\begin{eqnarray*}
  && \partial_t^\alpha \left(w(x,t)-\underline{w}(x,t)\right)+(-\Delta)^s\left(w(x,t)-\underline{w}(x,t)\right) \\
   &\geq& -\varepsilon -\delta\Phi(x)\partial_t^\alpha \eta(t)-(-\Delta)^s\left(w(x,t)\chi_{D\cup D^\lambda}(x)\right)-\delta\eta(t)(-\Delta)^s\left(\phi(x)-\phi_\lambda(x)\right)\\
   &\geq&-\varepsilon+C_{n,s}\int_{D\cup D^\lambda}\frac{w(y,t)}{|x-y|^{n+2s}}\operatorname{d}\!y-\frac{C\delta}{r^{2s}}-C_{n,s}\delta\int_{B_r((x^0)^\lambda)}\frac{1}{|x-y|^{n+2s}}\operatorname{d}\!y\\
   &\geq&-\varepsilon+C_{n,s}C_0\int_{D}\left[\frac{1}{|x-y|^{n+2s}}-\frac{1}{|x-y^\lambda|^{n+2s}}\right]\operatorname{d}\!y-\frac{C\delta}{r^{2s}}\\
   &\geq&-\varepsilon+C_2-\frac{C\delta}{r^{2s}}\geq 0.
\end{eqnarray*}
For $(x,t)\in  (\Sigma_\lambda\setminus B_r(x^0))\times(t_0-r^{\frac{2s}{\alpha}},t_0+r^{\frac{2s}{\alpha}}]$,  it follows from the exterior condition in \eqref{AEA2} that
$$w(x,t)-\underline{w}(x,t)=w(x,t)-w(x,t)\chi_{D\cup D^\lambda}(x)\geq0.$$
Finally, taking into account $(x,t)\in  B_r(x^0)\times(-\infty,t_0-r^{\frac{2s}{\alpha}}]$, we apply the initial condition in \eqref{AEA2} to derive
$$w(x,t)-\underline{w}(x,t)=w(x,t)\geq0.$$
Through the above arguments, we have deduced that
\begin{equation*}
\begin{cases}
\partial_t^\alpha \left(w(x,t)-\underline{w}(x,t)\right)+(-\Delta)^s\left(w(x,t)-\underline{w}(x,t)\right)\geq0, & \mbox{in} \,\, B_r(x^0)\times(t_0-r^{\frac{2s}{\alpha}},t_0+r^{\frac{2s}{\alpha}}], \\
w(x,t)-\underline{w}(x,t)\geq 0, & \mbox{in} \,\, (\Sigma_\lambda\setminus B_r(x^0))\times(t_0-r^{\frac{2s}{\alpha}},t_0+r^{\frac{2s}{\alpha}}],\\
w(x,t)-\underline{w}(x,t)\geq 0, & \mbox{in} \,\, B_r(x^0)\times(-\infty,t_0-r^{\frac{2s}{\alpha}}],\\
w(x,t)-\underline{w}(x,t)=-\left(w(x^\lambda,t)-\underline{w}(x^\lambda,t)\right), & \mbox{in} \,\, \Sigma_\lambda\times\mathbb{R}.
\end{cases}
\end{equation*}
Then by virtue of the maximum principle for antisymmetric functions established in Theorem \ref{MPA}\,, we obtain
\begin{equation*}
w(x,t)\geq\underline{w}(x,t) \,\,\mbox{in}\,\,  B_r(x^0)\times(t_0-r^{\frac{2s}{\alpha}},t_0+r^{\frac{2s}{\alpha}}].
\end{equation*}
Hence, we conclude that $\underline{w}(x,t)$ is
a sub-solution of $w(x,t)$ in $B_r(x^0)\times(t_0-r^{\frac{2s}{\alpha}},t_0+r^{\frac{2s}{\alpha}}]$. As a consequence, we have
\begin{equation*}
 w(x^0,t_0)\geq \underline{w}(x^0,t_0)=\delta
\phi(x^0)\eta(t_0)=\delta=:C_1>0.
\end{equation*}
This completes the proof of Theorem \ref{AEA}.
\end{proof}

\section{Monotonicity in a half space}\label{4}
Based on the maximum principles and the averaging effects introduced in the previous section, we establish the monotonicity of positive solutions for nonlocal parabolic problem \eqref{model} in a half space by virtue of the direct method of moving planes.
For readers' convenience, we first provide a sketch of proof.

We proceed in three steps.

In the first step, we argue that for $\lambda > 0$ sufficiently close to $0$, i.e. when $\Omega_\lambda$ is a narrow region, it holds
\begin{equation}
 w_\lambda(x,t)  \geq 0,  \, \,\mbox{in}\,\, \Omega_\lambda\times\mathbb{R}.
 \label{A121}
 \end{equation}
This provides a starting point for moving the plane $T_\lambda$ to the right along the $x_{1}$-axis.

In the next step, we move the plane $T_\lambda$ continuously to the right along the $x_{1}$-axis as long as inequality (\ref{A121}) is valid to its limiting position.
We want to show that the plane can be moved all the way to positive infinity.  Otherwise,
we will apply the maximum principle for antisymmetric functions established in Theorem \ref{MPAF} and subtly employ the averaging effects acquired in Section \ref{3} twice to derive a contradiction along a sequence of approximate minimum points.

In the final step, we verify that $u(x,t)$ is strictly increasing with respect to $x_{1}$ in $\mathbb{R}^n_+$ for any $t\in\mathbb{R}$.
To this end, it suffice to derive  a strong maximum principle based on \eqref{A121} for the dual nonlocal operator $\partial_t^\alpha+(-\Delta)^s$ to arrive at
\begin{equation*}
  w_{\lambda}(x,t)>0, \,\,\mbox{in}\,\,\Omega_\lambda\times\mathbb{R}\,\, \mbox{for any}\,\, \lambda>0.
\end{equation*}

Now we show the details.

\begin{proof}[\bf Proof of Theorem \ref{MainR}\,.] By a direct calculation, we have
\begin{equation}\label{MR2}
\left\{
\begin{array}{ll}
    \partial_t^\alpha w_\lambda(x,t)+(-\Delta)^sw_\lambda(x,t)=C_\lambda(x,t)w_\lambda(x,t),~   &(x,t) \in  \Omega_\lambda\times\mathbb{R}  , \\
  w_\lambda(x,t)\geq 0 , ~ &(x,t)  \in (\Sigma_\lambda\backslash\Omega_\lambda)\times\mathbb{R},\\
   w_\lambda(x,t)=- w_\lambda(x^\lambda,t), ~ &(x,t)  \in \Sigma_\lambda\times\mathbb{R},\\
\end{array}
\right.
\end{equation}
where the coefficient function
\begin{equation*}
  C_\lambda(x,t)=\int_0^1 f'\left(\varsigma u_\lambda(x,t)+(1-\varsigma)u(x,t)\right)\operatorname{d}\!\varsigma\leq C,
\end{equation*}
which is ensured by $f'$ is bounded
from above. In order to prove the strict monotonicity of $u(x,t)$ with respect to $x_1$, it suffices to show that the antisymmetric
function $w_\lambda(x,t)>0$ in $\Omega_\lambda\times\mathbb{R}$ for any $\lambda>0$. Now we divide the proof into three steps.

\noindent \textup{\textbf{Step 1.}} Start moving the plane $T_{\lambda}$ from $x_1=0$ to the right along the $x_{1}$-axis. If the positive $\lambda$ is small enough, then the assumptions in Theorem \ref{MainR} guarantee that we can apply Theorem \ref{NRP} to the problem \eqref{MR2} to derive
\begin{equation}\label{MR3}
  w_{\lambda}(x,t)\geq0 \,\, \mbox{in}\,\, \Omega_{\lambda}\times\mathbb{R},
\end{equation}
where $\Omega_\lambda$ is a narrow region for sufficiently small $\lambda>0$.
Inequality \eqref{MR3} provides a starting point to move the plane $T_\lambda$.

\noindent \textup{\textbf{Step 2.}} In this step, we continue to move the plane $T_\lambda$ to the right along the $x_{1}$-axis as long as \eqref{MR3} is valid to its limiting position. Let
 \begin{equation*}
 \lambda_0 := \sup \left\{ \lambda \mid w_\mu (x,t) \geq 0,\,\, (x,t) \in \Sigma_\mu\times\mathbb{R}\,\,\mbox{for any} \,\, \mu \leq \lambda \right\},
 \end{equation*}
We will show that
\begin{equation}\label{MR4}
 \lambda_0=+\infty.
\end{equation}
Otherwise, if $0<\lambda_0<+\infty$, then by its definition,  there exists a sequence of $\lambda_k>\lambda_0$ such that $\lambda_k\rightarrow \lambda_0$ as $k\rightarrow\infty$, along which
\begin{equation*}
  \Sigma_{\lambda_k}^-\times\mathbb{R}:=\{(x,t)\in \Sigma_{\lambda_k}\times\mathbb{R}\mid w_{\lambda_k}(x,t)<0\}
\end{equation*}
is nonempty and $\displaystyle\inf_{\Sigma_{\lambda_k}\times\mathbb{R}}w_{\lambda_k}(x,t)<0$. We first show that
\begin{equation}\label{MR4-3}
  \inf_{\Sigma_{\lambda_k}\times\mathbb{R}}w_{\lambda_k}(x,t)\rightarrow 0\,\,\mbox{as}\,\,k\rightarrow\infty.
\end{equation}
If not, then there exists a uniformly positive constant $M$ such that $$\displaystyle\inf_{\Sigma_{\lambda_k}\times\mathbb{R}}w_{\lambda_k}(x,t)<-M<0.$$ From this, there exists a sequence $\{(x^k,t_k)\}\subset\Sigma_{\lambda_k}\times\mathbb{R}$ such that
\begin{equation}\label{MR4-1}
w_{\lambda_k}(x^k,t_k)\leq-M<0.
\end{equation}
If $x^k$ is between $T_{\lambda_0}$ and $T_{\lambda_k}$, then by virtue of $\lambda_k\rightarrow \lambda_0$ as $k\rightarrow\infty$, we have $|x^k-(x^k)^{\lambda_k}|\rightarrow 0$ as $k\rightarrow\infty$. By $u(x,t)$ is uniformly continuous with respect to $x$, we derive
\begin{equation*}
  w_{\lambda_k}(x^k,t_k)=u((x^k)^{\lambda_k},t_k)-u(x^k,t_k)\rightarrow 0\,\,\mbox{as}\,\, k\rightarrow\infty,
\end{equation*}
which contradicts with \eqref{MR4-1}.
In the other case, if $x^k\in \Omega_{\lambda_0}$, then combining the uniform continuity of $u(x,t)$ in $x$ with $\lambda_k\rightarrow \lambda_0$ as $k\rightarrow\infty$ again, we have
\begin{equation*}
  w_{\lambda_k}(x^k,t_k)-w_{\lambda_0}(x^k,t_k)=u((x^k)^{\lambda_k},t_k)-u((x^k)^{\lambda_0},t_k)\rightarrow 0\,\,\mbox{as}\,\, k\rightarrow\infty.
\end{equation*}
While in terms of \eqref{MR4-1} and $w_{\lambda_0}(x^k,t_k)\geq 0$, we must derive a contradiction that $$w_{\lambda_k}(x^k,t_k)-w_{\lambda_0}(x^k,t_k)\leq-M<0.$$
Hence, we deduce that
\begin{equation}\label{MR4-2}
  \inf_{\Sigma_{\lambda_k}\times\mathbb{R}}w_{\lambda_k}(x,t)=:-m_k\rightarrow0\,\,\mbox{as}\,\, k\rightarrow\infty.
\end{equation}

In the sequel, we denote the sequence $$q_k:=\displaystyle\sup_{\Sigma_{\lambda_k}^-\times\mathbb{R}}C_{\lambda_k}(x,t),$$
there are the following two possibilities.

\noindent \textup{\textbf{Case 1.}} If $q_k\leq \varepsilon_k\rightarrow 0$ as $k\rightarrow\infty$, then it infers that $C_{\lambda_k}(x,t)\leq C$ in $\Sigma_{\lambda_k}^-\times\mathbb{R}$ for sufficiently large $k$ and for any given small positive constant $C$.
Hence, by virtue of Theorem \ref{MPAF} to the problem \eqref{MR2} with $\lambda=\lambda_k$, we conclude that $$w_{\lambda_k}(x,t)\geq 0\,\, \mbox{in} \,\,\Sigma_{\lambda_k}\times\mathbb{R}$$ for
sufficiently large $k$, which is a contradiction with the definition of $\lambda_k$.

\noindent \textup{\textbf{Case 2.}} If positive sequence $q_k\nrightarrow 0$ as $k\rightarrow\infty$, then there exist a positive constant $\delta_0$ and a subsequence of $\{q_k\}$ (still denoted by $\{q_k\}$) such that $q_k\geq\delta_0>0$.
In this case, by using the condition $f'(0)\leq0$ and \eqref{MR4-2}, we can deduce that there exist a positive constant $\varepsilon_0$ and a sequence $\{(x^k,t_k)\}\subset\Sigma^-_{\lambda_k}\times\mathbb{R}$ such that $$u(x^k,t_k)\geq\varepsilon_0>0$$ and
\begin{equation*}
  w_{\lambda_k}(x^k,t_k)=-m_k+m_k^2<0.
\end{equation*}
Then a combination of $u(x,t)=0$ in $(\mathbb{R}^n\setminus\mathbb{R}^n_+)\times\mathbb{R}$ and the continuity of $u(x,t)$ yields that there exists a small radius $r_0>0$ independent of $k$, such that
\begin{equation}\label{MR5}
  u(x,t)\geq \frac{\varepsilon_0}{2}>0 \,\,\mbox{in}\,\, B_{r_0}(x^k)\times(t_k-r_0^{\frac{2s}{\alpha}},t_k+r_0^{\frac{2s}{\alpha}}]\subset\mathbb{R}^n_+\times \mathbb{R}.
\end{equation}

Now we claim that  $\delta_k:=\operatorname{dist}(x^k,T_{\lambda_k})=\lambda_k-x_1^k$ is bounded away from zero for sufficiently large $k$. If not, then $\delta_k\rightarrow 0$ as $k\rightarrow\infty$. Let
$$v_k(x,t):=w_{\lambda_k}(x,t)-m_k^2\eta_k(x,t),$$
where the sequence of smooth cut-off functions
$$\eta_k(x,t):=\eta(\frac{x-x^k}{\delta_k},\frac{t-t_k}{\delta_k^{\frac{2s}{\alpha}}})\in C_0^\infty\left(B_{\delta_k}(x^k)\times (t_k-\delta_k^{\frac{2s}{\alpha}},t_k+\delta_k^{\frac{2s}{\alpha}})\right)$$
satisfies
$$0\leq\eta_k \leq1, \,\,\mbox{and} \,\,\eta_k(x,t)\equiv1\,\, \mbox{in} \,\, B_{\frac{\delta_k}{2}}(x^k)\times (t_k-\frac{\delta_k^{\frac{2s}{\alpha}}}{2},t_k+\frac{\delta_k^{\frac{2s}{\alpha}}}{2}).$$
We denote the parabolic cylinder $$Q_{\delta_k}(x^k,t_k):=B_{\delta_k}(x^k)\times (t_k-\delta_k^{\frac{2s}{\alpha}},t_k+\delta_k^{\frac{2s}{\alpha}}),$$ thereby a direct calculation infers that
\begin{equation*}
  v_k(x^k,t_k)=w_{\lambda_k}(x^k,t_k)-m_k^2\eta_k(x^k,t_k)=-m_k+m_k^2-m_k^2=-m_k,
\end{equation*}
 and
\begin{equation*}
  v_k(x,t)=w_{\lambda_k}(x,t)\geq-m_k \,\, \mbox{for}\,\, (x,t)\in Q^c_{\delta_k}(x^k,t_k)\cap(\Sigma_{\lambda_k}\times \mathbb{R}).
\end{equation*}
Hence, there exists a point $(\bar{x}^k,\bar{t}_k)\in Q_{\delta_k}(x^k,t_k)$ such that
\begin{equation*}
-m_k-m_k^2  \leq v_k(\bar{x}^k,\bar{t}_k)=\inf_{\Sigma_{\lambda_k}\times \mathbb{R}}v_k(x,t)\leq-m_k.
\end{equation*}
Moreover, it follows from the definition of $v_k$ that $$-m_k\leq w_{\lambda_k}(\bar{x}^k,\bar{t}_k)\leq -m_k+m_k^2<0.$$
On such minimum point $(\bar{x}^k,\bar{t}_k)$ of $v_k(x,t)$, we estimate
\begin{equation*}
 \partial_t^\alpha v_k(\bar{x}^k,\bar{t}_k)=C_{\alpha}\int_{-\infty}^{\bar{t}_k} \frac{v_k(\bar{x}^k,\bar{t}_k)-v_k(\bar{x}^k,\tau)}{(\bar{t}_k-\tau)^{1+\alpha}}\operatorname{d}\!\tau\leq 0,
\end{equation*}
and
\begin{eqnarray*}
  &&(-\Delta)^sv_k(\bar{x}^k,\bar{t}_k) \\
  &=&C_{n,s} P.V. \int_{\Sigma_{\lambda_k}} \frac{v_k(\bar{x}^k,\bar{t}_k)-v_k(y,\bar{t}_k)}{|\bar{x}^k-y|^{n+2s}}\operatorname{d}\!y+C_{n,s} \int_{\Sigma_{\lambda_k}^c} \frac{v_k(\bar{x}^k,\bar{t}_k)-v_k(y,\bar{t}_k)}{|\bar{x}^k-y|^{n+2s}}\operatorname{d}\!y  \\
  &\leq&  C_{n,s}\int_{\Sigma_{\lambda_k}} \frac{2v_k(\bar{x}^k,\bar{t}_k)+m_k^2\eta_k(y,\bar{t}_k)}{|\bar{x}^k-y^{\lambda_k}|^{n+2s}}\operatorname{d}\!y\\
   &\leq&\frac{Cv_k(\bar{x}^k,\bar{t}_k)}{\delta_k^{2s}}+\frac{Cm_k^2}{\delta_k^{2s}}\\
   &\leq&-\frac{Cm_k}{\delta_k^{2s}}+\frac{Cm_k^2}{\delta_k^{2s}}.
\end{eqnarray*}
While from the equation in \eqref{MR2} and Corollary \ref{coro1}\,, we derive
\begin{eqnarray*}
   && \partial_t^\alpha v_k(\bar{x}^k,\bar{t}_k)+(-\Delta)^sv_k(\bar{x}^k,\bar{t}_k) \\
  &=& \partial_t^\alpha w_{\lambda_k}(\bar{x}^k,\bar{t}_k)+(-\Delta)^sw_{\lambda_k}(\bar{x}^k,\bar{t}_k)-m_k^2\partial_t^\alpha \eta_k(\bar{x}^k,\bar{t}_k)-m_k^2(-\Delta)^s\eta_k(\bar{x}^k,\bar{t}_k) \\
  &\geq&C_{\lambda_k}(\bar{x}^k,\bar{t}_k)w_{\lambda_k}(\bar{x}^k,\bar{t}_k)-\frac{Cm_k^2}{\delta_k^{2s}}\\
 &\geq&-Cm_k-\frac{Cm_k^2}{\delta_k^{2s}},
\end{eqnarray*}
where the last line is ensured by $f'$ with an upper bound. Then it follows from the above estimates that
\begin{equation*}
  -\frac{Cm_k}{\delta_k^{2s}}\geq-Cm_k-\frac{Cm_k^2}{\delta_k^{2s}}.
\end{equation*}
Multiplying both sides by $\frac{\delta_k^{2s}}{-m_k}$ and combining \eqref{MR4-2} with the assumption $\displaystyle\lim_{k\rightarrow\infty}\delta_k=0$, we obtain
\begin{equation*}
  0<C\leq C\delta_k^{2s}+Cm_k\rightarrow 0 \,\,\mbox{as}\,\, k\rightarrow\infty.
\end{equation*}
This contradiction indicates that $\delta_k$ is bounded away from zero for sufficiently large $k$. Furthermore, since $\lambda_k\rightarrow \lambda_0$ as $k\rightarrow\infty$, then we deduce that there exists a subsequence of $\{x^k,t_k\}$ (still denoted by $\{x^k,t_k\}$) such that $\{(x^k,t_k)\}\subset \Sigma_{\lambda_0}\times\mathbb{R}$ and $\operatorname{dist}\{x^k,T_{\lambda_0}\}\geq\delta_0>0$.
As a consequence of \eqref{MR5}, we further select a radius $r_1:=\min\{r_0,\delta_0\}$ such that
\begin{equation}\label{MR5-1}
  u(x,t)\geq \frac{\varepsilon_0}{2}>0 \,\,\mbox{in}\,\, B_{r_1}(x^k)\times(t_k-r_1^{\frac{2s}{\alpha}},t_k+r_1^{\frac{2s}{\alpha}}]\subset\Omega_{\lambda_0}\times \mathbb{R}.
\end{equation}

Before continuing, we illustrate the ideas of the proof based on Figure 4 below for ease of understanding. Our purpose is to show that $w_{\lambda_k}(x^k,t_k)>0$, which will contradict $(x^k,t_k)\in\Sigma^-_{\lambda_k}\times\mathbb{R}$. With this aim in mind, combining \eqref{MR5-1} with the averaging effects established in
Theorem \ref{AE}\,, we first claim that the solution $u(x,t)$ has a positive lower bound
in $ B_{r_1}(\bar{x}^k)\times(t_k-r_1^{\frac{2s}{\alpha}},t_k+r_1^{\frac{2s}{\alpha}}]$.
We next apply the exterior condition in \eqref{model} to derive that $u(x,t)$ has a smaller upper bound in $\left(B_{r_2}(\hat{x}^k)\cap\mathbb{R}^n_+\right)\times(t_k-r_1^{\frac{2s}{\alpha}},t_k+r_1^{\frac{2s}{\alpha}}]$.
It follows that the antisymmetric function $w_{\lambda_0}(x,t)$ is positively bounded away from zero in
$\left(B_{r_2}(\hat{x}^k)\cap\mathbb{R}^n_+\right)\times(t_k-r_1^{\frac{2s}{\alpha}},t_k+r_1^{\frac{2s}{\alpha}}]$.
We further use the averaging effects demonstrated in Theorem \ref{AEA} to deduce that
$w_{\lambda_0}(x,t)\geq \varepsilon_2>0$  in $B_{\frac{r_1}{2}}(x^k)\times\left(t_k-(\frac{r_1}{2})^{\frac{2s}{\alpha}},t_k+(\frac{r_1}{2})^{\frac{2s}{\alpha}}\right]$.
Finally, a combination of the continuity of $w_\lambda(x,t)$ with respect to $\lambda$ and $\lambda_k\rightarrow \lambda_0$ as $k\rightarrow\infty$ yields that
$ w_{\lambda_k}(x^k,t_k)\geq\frac{ \varepsilon_2}{2}>0$ for sufficiently large $k$.

\begin{center}
\begin{tikzpicture}[node distance = 0.3cm]
\draw [->, semithick] (-3,-1) -- (-3,3) node[above] {$x'$};
\path (-2.75,-0.75) node [ font=\fontsize{10}{10}\selectfont] {$T_0$};
\draw  [->,thick](-4,0)--(5,0) node [anchor=north west] {$x_1$};
\draw [semithick] (0,-1) -- (0,3);
\path (0.35,-0.75) node  [ font=\fontsize{10}{10}\selectfont] {$T_{\lambda_0}$};
\draw [semithick] [blue][dashed](0.65,-1) -- (0.65,3);
\path (0.95,-0.75) node [ font=\fontsize{10}{10}\selectfont] {$T_{\lambda_k}$};
\draw [semithick] (3,-1) -- (3,3.5);
\path (3.5,-0.75) node [ font=\fontsize{10}{10}\selectfont] {$T_{2\lambda_0}$};
\draw (-1.25,2) ellipse[x radius=1cm, y radius=1cm];
\draw (-1.25,2) [orange] ellipse[x radius=0.5cm, y radius=0.5cm];
\path (-1.25,2)[very thick,fill=red]  circle(1.3pt) node [ font=\fontsize{9}{9}\selectfont] at (-1.1,1.95) {$x^k$};
\draw (-1.25,2)-- (-2.25,2);
\path (-1.9,2.15) node [ font=\fontsize{9}{9}\selectfont] {$r_1$};
\fill[blue!50] (3,2.45) arc(90:270:0.45);
\path (3,2)[very thick,fill=red]  circle(1.3pt) node [ font=\fontsize{9}{9}\selectfont] at (3.2,1.8) {$\bar{x}^k$};
\draw (3,2) ellipse[x radius=1cm, y radius=1cm];
\draw (3,2)[green] ellipse[x radius=0.5cm, y radius=0.5cm];
\draw (2,2)-- (3,2);
\path (2.3,2.15) node [ font=\fontsize{9}{9}\selectfont] {$r_1$};
\fill[blue!50] (-3,1.55) arc(-90:90:0.45);
\path (-3,2)[very thick,fill=red]  circle(1.3pt) node [ font=\fontsize{9}{9}\selectfont] at (-2.8,1.8) {$\hat{x}^k$};
\draw (-3,2) ellipse[x radius=0.45cm, y radius=0.45cm];
\draw (-3,2)-- (-2.55,2);
\path (-2.8,2.15) node [ font=\fontsize{9}{9}\selectfont] {$r_2$};
\node [below=1.5cm, align=flush center,text width=10cm] at (0,-0.1)
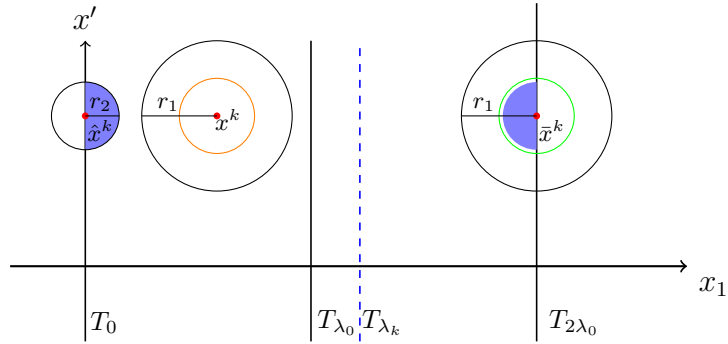
        {Figure 4. The positional relationship between the balls. };
\end{tikzpicture}
\end{center}

We now carry out the details of the proof.

Let $\bar{x}^k=(2\lambda_0,(x^k)')$, since $$\operatorname{dist}(T_{\lambda_0},T_{2\lambda_0})=\lambda_0>2r_1,$$
then we have $\overline{B_{r_1}(x^k)}\cap B_{2r_1}(\bar{x}^k)=\varnothing$.
Here the main purpose is to claim that there exists a positive constant $\varepsilon_1=\varepsilon_1(\alpha,n,s,\varepsilon_0,\lambda_0,r_1)$ such that
\begin{equation}\label{MR6}
  u(x,t)\geq \varepsilon_1>0 \,\,\mbox{in}\,\, B_{r_1}(\bar{x}^k)\times(t_k-r_1^{\frac{2s}{\alpha}},t_k+r_1^{\frac{2s}{\alpha}}].
\end{equation}
If not, that is to say
\begin{equation}\label{MR7}
  u(x,t)< \varepsilon\,\,\mbox{ in} \,\, B_{r_1}(\bar{x}^k)\times(t_k-r_1^{\frac{2s}{\alpha}},t_k+r_1^{\frac{2s}{\alpha}}]\,\,\mbox{ for any} \,\,\varepsilon>0,
\end{equation}
then applying the condition that $f$ is a $C^1$ function, we have
$$|f(u(x, t))-f(0)|\leq C|(u(x, t))-0|<C\varepsilon\,\,\mbox{in}\,\, B_{r_1}(\bar{x}^k)\times(t_k-r_1^{\frac{2s}{\alpha}},t_k+r_1^{\frac{2s}{\alpha}}].$$ Furthermore, it follows from the assumption $f(0)\geq 0$ that
\begin{equation*}
  f(u(x, t))> -C\varepsilon  \,\, \mbox{in}\,\, B_{r_1}(\bar{x}^k)\times(t_k-r_1^{\frac{2s}{\alpha}},t_k+r_1^{\frac{2s}{\alpha}}]\,\,\mbox{for any}\,\, \varepsilon>0.
\end{equation*}
Then by virtue of  \eqref{model}, \eqref{MR5-1} and the averaging effects for the nonlocal operators established in
Theorem \ref{AE}\,, and combining with the continuity of $u(x,t)$, we derive
$$ u(x,t)\geq \varepsilon_1>0\,\,\mbox{in}\,\, B_{\frac{r_1}{2}}(\bar{x}^k)\times\left(t_k-(\frac{r_1}{2})^{\frac{2s}{\alpha}},t_k+(\frac{r_1}{2})^{\frac{2s}{\alpha}}\right],$$
where the positive constant $\varepsilon_1=\varepsilon_1(\alpha,n,s,\varepsilon_0,\lambda_0,r_1)$. It apparently contradicts \eqref{MR7}, which verifies \eqref{MR6}.

Next, let $\hat{x}^k=(0,(x^k)')$, using the exterior condition
$u(x,t)\equiv0$ in $(\mathbb{R}^n\setminus\mathbb{R}^n_+)\times\mathbb{R}$ and the continuity of $u(x,t)$ again, we deduce that there exists some positive small radius $r_2<\frac{r_1}{2}$ independent of $k$, such that
\begin{equation}\label{MR8}
  u(x,t)\leq \frac{\varepsilon_1}{2} \,\,\mbox{in}\,\, \left(B_{r_2}(\hat{x}^k)\cap\mathbb{R}^n_+\right)\times(t_k-r_1^{\frac{2s}{\alpha}},t_k+r_1^{\frac{2s}{\alpha}}],
\end{equation}
where $\varepsilon_1$ is given in \eqref{MR6}. Note that for any point $x\in B_{r_2}(\hat{x}^k)\cap\mathbb{R}^n_+$, its reflection point $x^{\lambda_0}$ with respect to the plane $T_{\lambda_0}$ belongs to $B_{r_2}(\bar{x}^k)\cap\Sigma_{2\lambda_0}\subset B_{\frac{r_1}{2}}(\bar{x}^k)$ by $r_2<\frac{r_1}{2}$.
Then a combination of \eqref{MR6} and \eqref{MR8} yields that
\begin{equation}\label{MR9-1}
  w_{\lambda_0}(x,t)=u(x^{\lambda_0},t)-u(x,t)\geq\varepsilon_1-\frac{\varepsilon_1}{2} =\frac{\varepsilon_1}{2}>0  \,\,\mbox{in}\,\,  \left(B_{r_2}(\hat{x}^k)\cap\mathbb{R}^n_+\right)\times(t_k-r_1^{\frac{2s}{\alpha}},t_k+r_1^{\frac{2s}{\alpha}}]\,.
\end{equation}

The ultimate aim is to demonstrate that there exists a positive constant $\varepsilon_2=\varepsilon_2(\alpha,n,s,\varepsilon_0,\lambda_0,r_1)$, such that
\begin{equation}\label{MR9}
  w_{\lambda_0}(x,t)\geq \varepsilon_2>0  \,\,\mbox{for}\,\, (x,t)\in B_{\frac{r_1}{2}}(x^k)\times\left(t_k-(\frac{r_1}{2})^{\frac{2s}{\alpha}},t_k+(\frac{r_1}{2})^{\frac{2s}{\alpha}}\right]\,.
\end{equation}
Otherwise, there holds that
\begin{equation}\label{MR10}
  w_{\lambda_0}(x,t)<\varepsilon  \,\,\mbox{in}\,\,  B_{\frac{r_1}{2}}(x^k)\times\left(t_k-(\frac{r_1}{2})^{\frac{2s}{\alpha}},t_k+(\frac{r_1}{2})^{\frac{2s}{\alpha}}\right] \,\,\mbox{for any} \,\,\varepsilon>0.
\end{equation}
Considering \eqref{MR2} and the definition of $\lambda_0$, we obtain
\begin{equation*}
\left\{
\begin{array}{ll}
    \partial_t^\alpha w_{\lambda_0}(x,t)+(-\Delta)^sw_{\lambda_0}(x,t)=C_{\lambda_0}(x,t)w_{\lambda_0}(x,t),   \,\,\mbox{in}\,\,  B_{\frac{r_1}{2}}(x^k)\times\left(t_k-(\frac{r_1}{2})^{\frac{2s}{\alpha}},t_k+(\frac{r_1}{2})^{\frac{2s}{\alpha}}\right]   , \\
 w_{\lambda_0}(x,t)\geq 0 ,  \,\,\qquad\qquad\mbox{in}\,\, \left(\Sigma_{\lambda_0}\setminus B_{\frac{r_1}{2}}(x^k)\right)\times\left(t_k-(\frac{r_1}{2})^{\frac{2s}{\alpha}},t_k+(\frac{r_1}{2})^{\frac{2s}{\alpha}}\right],\\
 w_{\lambda_0}(x,t)\geq 0 ,  \,\,\qquad\qquad\mbox{in}\,\,  B_{\frac{r_1}{2}}(x^k)\times\left(-\infty,t_k-(\frac{r_2}{2})^{\frac{2s}{\alpha}}\right].
\end{array}
\right.
\end{equation*}
Meanwhile, if follows from \eqref{MR10} and $f\in C^1$ that
\begin{equation*}
 C_{\lambda_0}(x,t)w_{\lambda_0}(x,t)=f(u_{\lambda_0}(x, t))-f(u(x, t))> -C\varepsilon  \,\, \mbox{in}\,\, B_{\frac{r_1}{2}}(x^k)\times\left(t_k-(\frac{r_1}{2})^{\frac{2s}{\alpha}},t_k+(\frac{r_1}{2})^{\frac{2s}{\alpha}}\right],
\end{equation*}
for any $\varepsilon>0$. Thereby using \eqref{MR9-1} and the averaging effects established in Theorem \ref{AEA} for the antisymmetric function $w_{\lambda_0}(x,t)$, and combining with
the continuity of $u(x,t)$, we conclude that
\begin{equation*}
  w_{\lambda_0}(x,t)\geq \varepsilon_2>0  \,\,\mbox{in}\,\, B_{\frac{r_1}{4}}(x^k)\times\left(t_k-(\frac{r_1}{4})^{\frac{2s}{\alpha}},t_k+(\frac{r_1}{4})^{\frac{2s}{\alpha}}\right]
\end{equation*}
for some positive constant $\varepsilon_2=\varepsilon_2(\alpha,n,s,\varepsilon_0,\lambda_0,r_1)$, which contradicts the assumption \eqref{MR10}. Hence, we verify that \eqref{MR9} is valid.
 Moreover, applying \eqref{MR9}, and combining the continuity of $w_\lambda(x,t)$ in $\lambda$ with $\lambda_k\rightarrow \lambda_0$ as $k\rightarrow\infty$, we finally derive
\begin{equation*}
  w_{\lambda_k}(x,t)\geq\frac{ \varepsilon_2}{2}>0  \,\, \mbox{in} \,\,B_{\frac{r_1}{2}}(x^k)\times\left(t_k-(\frac{r_1}{2})^{\frac{2s}{\alpha}},t_k+(\frac{r_1}{2})^{\frac{2s}{\alpha}}\right] \,\,\mbox{for sufficiently large}\,\, k,
\end{equation*}
which means that $ w_{\lambda_k}(x^k,t_k)\geq\frac{ \varepsilon_2}{2}>0$ for sufficiently large $k$. Hence, it contradicts the assumption that the sequence $\{(x^k,t_k)\}\subset\Sigma^-_{\lambda_k}\times\mathbb{R}$, therefore we must have $\lambda_0=+\infty$.

\noindent \textup{\textbf{Step 3.}} In this final step, we prove that $u(x,t)$ is strictly increasing with respect to $x_{1}$ in $\mathbb{R}^n_+$ for any $t\in\mathbb{R}$.
By virtue of Step 1 and Step 2, we have deduced that
\begin{equation*}
  w_{\lambda}(x,t)\geq0 \,\,\mbox{in}\,\,\Sigma_\lambda\times\mathbb{R}\,\, \mbox{for any}\,\, \lambda>0.
\end{equation*}
In fact, it suffices to show that
\begin{equation}\label{MR11}
  w_{\lambda}(x,t)>0 \,\,\mbox{in}\,\,\Omega_\lambda\times\mathbb{R}\,\, \mbox{for any}\,\, \lambda>0,
\end{equation}
then we can conclude that $u(x,t)$ is strictly increasing with respect to $x_{1}$ in $\mathbb{R}^n_+$ for any $t\in\mathbb{R}$. Suppose \eqref{MR11} is violated, then there exist a fixed $\lambda_0>0$ and a point $(x^0,t_0)\in \Omega_{\lambda_0}\times\mathbb{R}$ such that
\begin{equation*}
  w_{\lambda_0}(x^0,t_0)=\min_{\Sigma_{\lambda_0}\times\mathbb{R}} w_{\lambda_0}(x,t)=0.
\end{equation*}
Since $w_{\lambda_0}(x,t_0)$ is not identically zero in $\Sigma_{\lambda_0}$ due to the exterior condition $u(x,t)\equiv 0$ in $(\mathbb{R}^n  \backslash \mathbb{R}^n_+) \times\mathbb{R}$ and the interior positivity of the solution $u(x,t)$ in $ \mathbb{R}^n_+ \times\mathbb{R}$, then through a straightforward calculation, we obtain
\begin{eqnarray*}
    &&\partial_t^\alpha w_{\lambda_0}(x^0,t_0)+(-\Delta)^sw_{\lambda_0}(x^0,t_0)\\
    &=&C_{\alpha}\int_{-\infty}^{t_0} \frac{w_{\lambda_0}(x^0,t_0)-w_{\lambda_0}(x^0,\tau)}{(t_0-\tau)^{1+\alpha}}\operatorname{d}\!\tau+C_{n,s} P.V. \int_{\mathbb{R}^n} \frac{w_{\lambda_0}(x^0,t_0)-w_{\lambda_0}(y,t_0)}{|x^0-y|^{n+2s}}\operatorname{d}\!y\\
    &\leq&C_{n,s} P.V. \int_{\Sigma_{\lambda_0}} w_{\lambda_0}(y,t_0)\left(\frac{1}{|x^0-y^{\lambda_0}|^{n+2s}}-\frac{1}{|x^0-y|^{n+2s}}\right)\operatorname{d}\!y<0,
\end{eqnarray*}
which contradicts the equation
\begin{equation*}
    \partial_t^\alpha w_{\lambda_0}(x^0,t_0)+(-\Delta)^sw_{\lambda_0}(x^0,t_0)=f(u_{\lambda_0}(x^0,t_0))-f(u(x^0,t_0))=0.
\end{equation*}
This verifies \eqref{MR11}.

Finally, based on \eqref{MR11}, for each fixed $t \in \mathbb{R}$,  for any points $\bar{x}=(\bar{x}_1,x')$ and $\hat{x}=(\hat{x}_1,x')$ in $\mathbb{R}^{n}_+$ with $\bar{x}_1<\hat{x}_1$, if choosing $\lambda=\frac{\bar{x}_1+\hat{x}_1}{2}$ as illustrated in Figure 5 below, then we must have
\begin{equation*}
  0< w_\lambda(\bar{x},t)=u(\bar{x}^{\lambda},t)-u(\bar{x},t)=u(\hat{x},t)-u(\bar{x},t).
\end{equation*}
It implies that $u(x,t)$ is strictly increasing with respect to $x_{1}$ in $\mathbb{R}^n_+$ for any $t\in\mathbb{R}$.
\begin{center}
\begin{tikzpicture}[scale=0.8]
\draw [very thick]  [black] [->,very thick](-4,0)--(4,0) node [anchor=north west] {$x_1$};
\draw [very thick]  [black!80][->,very thick] (-2.5,-1)--(-2.5,3) node [black][ above] {$x'$};
\path node at (-2.7,-0.3) {$0$};
\draw [very thick]  [black!80] (1.5,-1)--(1.5,3) node [black][ above] {$T_{\lambda}$};
\path (0.5,1.5)[very thick,fill=red]  circle(1.8pt) node at (0.6,1.1) {$\bar{x}$};
\path (2.5,1.5)[very thick,fill=red]  circle(1.8pt) node at (2.6,1.1) {$\hat{x}$};
\draw [very thick] [dashed] [blue] (0.5,1.5)--(2.5,1.5);
\node [below=0.5cm, align=flush center,text width=12cm] at  (0,-0.5)
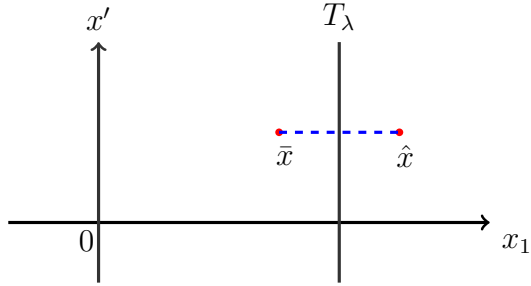
        {Figure 5. The positions of the points $\bar{x}$, $\hat{x}$ and the plane $T_\lambda$. };
\end{tikzpicture}
\end{center}

This completes the proof of Theorem \ref{MainR}.
\end{proof}

\section{Appendix}\label{5}
In this section, we provide some facts on the fractional time derivative that
are used repeatedly in establishing our main results.

\begin{lemma}\label{mlem1}
Let $\eta(t)\in C_0^\infty\left((-2,2)\right)$ be a smooth cut-off function, satisfying
$\eta(t)\equiv1$ in $[-1,1]$ and $0\leq\eta(t)\leq 1$. Then there exists a positive constant $C_0$ that depends only on $\alpha$ such that
$$|\partial_t^\alpha\eta(t)|\leq C_0 \,\, \mbox{for}\,\, t\in(-2,2).$$
\end{lemma}
\begin{proof}
By using the definitions of the fractional time derivative $\partial_t^\alpha$ and the smooth cut-off function $\eta(t)$, we directly compute
\begin{eqnarray*}
 |\partial_t^\alpha \eta(t)|&=&C_{\alpha}\left|\int_{-\infty}^t \frac{\eta(t)-\eta(\tau)}{(t-\tau)^{1+\alpha}}\operatorname{d}\!\tau\right|, \\
   &\leq&  C_{\alpha}\left|\int_{-\infty}^{-3} \frac{\eta(t)}{(t-\tau)^{1+\alpha}}\operatorname{d}\!\tau\right|+C_{\alpha}\left|\int_{-3}^{t} \frac{\eta(t)-\eta(\tau)}{(t-\tau)^{1+\alpha}}\operatorname{d}\!\tau\right|\\
   &\leq&\frac{C_\alpha}{\alpha}+CC_{\alpha}\left|\int_{-3}^{t} \frac{(t-\tau)}{(t-\tau)^{1+\alpha}}\operatorname{d}\!\tau\right|\\
   &\leq&\frac{C_\alpha}{\alpha}+\frac{CC_{\alpha}5^{1-\alpha}}{1-\alpha}=:C_0.
\end{eqnarray*}
Hence, we complete the proof of Lemma \ref{mlem1}\,.
\end{proof}
As a byproduct, we derive the following result by rescaling and translation.
\begin{corollary}\label{coro1}
For any $t_0\in \mathbb{R}$ and $r>0$, let $$\eta_0(t):=\eta\left(\frac{t-t_0}{r^{\frac{2s}{\alpha}}}\right)\in C_0^\infty\left((-2r^{\frac{2s}{\alpha}}+t_0,2r^{\frac{2s}{\alpha}}+t_0)\right),$$
then
$$|\partial_t^\alpha\eta_0(t)|\leq \frac{C_0}{r^{2s}} \,\, \mbox{for}\,\, t\in(-2r^{\frac{2s}{\alpha}}+t_0,2r^{\frac{2s}{\alpha}}+t_0),$$
where the smooth cut-off function $\eta(\cdot)$ and the positive constant $C_0$ are defined in Lemma \ref{mlem1}\,.
\end{corollary}

\section*{Acknowledgments}
The work of the first author is partially supported by MPS Simons foundation 847690.

This work of the second author is partially supported by the National Natural Science Foundation of China (NSFC Grant No.12101452).

\end{document}